\documentclass[a4paper,11pt,reqno,noindent]{amsart}
\usepackage[centertags]{amsmath}
\usepackage{amsfonts,amssymb,amsthm} 
\usepackage{hyperref}

 \hypersetup{
     pdfmenubar=false,        
     pdfnewwindow=true,      
     colorlinks=false,       
     linkcolor=blue,          
     citecolor=blue,        
     filecolor=magenta,      
     urlcolor=cyan           
 }
\usepackage{graphicx} 
\usepackage[numbers,sort]{natbib}
\usepackage[english]{babel}
\usepackage{newlfont}
\usepackage{color}
\usepackage[body={15cm,21.5cm},centering]{geometry} 
\usepackage{fancyhdr}
\usepackage{esint}
\usepackage{enumerate}

\newcommand{\Tr}{\mathrm{Tr}\,}

\usepackage[active]{srcltx}

\theoremstyle{plain}
\newtheorem{theorem}{Theorem}[section]
\newtheorem{lemma}[theorem]{Lemma}
\newtheorem{proposition}[theorem]{Proposition}

\newtheorem{definition}[theorem]{Definition}

\theoremstyle{definition}
\newtheorem{remark}[theorem]{Remark}

\newcommand{\sym}{\mathrm{sym}}

\newcommand{\loc}{\mathrm{loc}}

\newcommand{\G}{{\mathcal G}}
\newcommand{\dist}{\operatorname{dist}}
\newcommand{\supp}{\operatorname{supp}}

\newcommand{\e}{\varepsilon}

\newcommand\ecke{\mathop{\hbox{\vrule height 7pt width .3pt depth 0pt
\vrule height .3pt width 5pt depth 0pt}}\nolimits}

\newcommand{\R}{\mathbb{R}}

\newcommand{\N}{\mathbb{N}}

\renewcommand{\d}{\text{d}}
\renewcommand{\L}{\mathbb{L}}

\newcommand{\M}{\mathcal{M}}

\newcommand{\id}{\mathrm{Id}}

\newcommand{\F}{{\mathcal F}}

\newcommand{\curl}{{\mathrm{curl}\,}}

\renewcommand{\H}{\mathcal{H}}

\renewcommand{\L}{{\mathcal L}}
\newcommand\wto{\rightharpoonup}

\renewcommand{\div}{\mathrm{div}\,}
\newcommand{\Rs}{R^{\mathrm{sym}}}
\newcommand{\Rsym}{\R^{2\times 2}_{\mathrm{sym}}}

\newcommand{\cof}{\mathrm{cof}\,}

\title{Michell trusses  in two dimensions as a $\Gamma$-limit of  optimal design
  problems in  linear elasticity}
\date{\today}
\author[H. Olbermann] {Heiner Olbermann}
\address[Heiner Olbermann]{Universit\"at Leipzig, Germany}
\email{heiner.olbermann@math.uni-leipzig.de}

\begin{document}

\maketitle

\begin{abstract}
We reconsider the minimization of the compliance of a  two dimensional elastic body
with traction boundary conditions for a given weight.
It is well known how to  rewrite this optimal design problem as a
nonlinear variational problem. We take the limit of vanishing weight  by sending a suitable Lagrange
multiplier to infinity in the variational formulation. We show that the limit, in the sense of
$\Gamma$-convergence, is a certain Michell truss problem. This proves a conjecture by Kohn and Allaire.
\end{abstract}

\section{Introduction}

The aim of the present article is to derive a certain form of the Michell truss
problem from an optimal design problem in 
linear elasticity in two dimensions. The optimal design problem we consider is
the following classical question: Consider  an elastic body of  given  weight
loaded in plane stress. Which shape of the body minimizes the
 compliance (work done by the load)? There exist several different approaches to
 this problem; here we are going to be concerned with the ``homogenization 
 method'' that has been developed by 
 Lurie et al.~\cite{Lurie2}, 
 Gibiansky and Cherkaev \cite{MR1493041}, Murat and  Tartar \cite{MR807532},
 Kohn and Strang \cite{MR820342}, and others. 
The homogenization method rephrases the compliance optimization problem as a two-phase
design problem, and then enlarges the set of permissible designs via relaxation.
We refer the interested reader to
 Allaire's book \cite{MR1859696} for a more thorough account  of
 the method. 

\medskip

On a formal level, it has been noted by Allaire and Kohn \cite{allaire1993optimal}
 that the relaxed formulation of the problem  leads to a different
 variational problem in the limit of vanishing weight, namely a certain variant
 of the Michell
 truss problem (see also
 \cite{rozvany1987least,bendsoe1993michell}). This problem was first stated by Michell in
 1904 \cite{lviii185limits}. Michell trusses are elastic structures that consist of linear
 truss elements, each of which can withstand a certain tensile or compressive
 stress. The variational problem consists in finding the Michell truss of least
 weight that is admissible in the sense that it resists a given load. Michell
 himself already knew that this problem has no solution in general, and
 relaxation is required to assure existence of solutions. Since this day, the theory of Michell
 trusses has been very popular in the engineering and mathematics
 community. We refrain from attempting to give a comprehensive list of the
 relevant literature, and refer the reader to  \cite{hemp1973optimum,rozvany2012structural}.

\medskip

In the present article, we are going to cast the formal observations by Kohn and
Allaire into a rigorous statement. More precisely, we are going to prove that
the Michell truss problem is the limit of the compliance minimization problem in
 linear elasticity for vanishing weight in the sense of $\Gamma$-convergence
\cite{dal2012introduction}. 

\medskip

The plan of the paper is as follows:
Section  \ref{sec:stat-vari-probl} below is supposed to give the reader a quick
overview over the setting and main result. It consists of four Subsections: In
Sections \ref{sec:compl-minim-probl} and \ref{sec:mich-truss-probl}, we
are going to state the compliance minimization problem for  positive weight and the
Michell problem respectively. In Section \ref{sec:gamma-convergence}, our aim is   to give the reader a good idea of our main
result as quickly as possible, without too many 
 preparatory definitions. This is why we first state a special case
of our main theorem, Theorem \ref{thm:special}. In Section \ref{sec:derivation-variational}, we give a short
explanation of the form of the variational formulation of the compliance
minimization problem that we had presented in Section \ref{sec:compl-minim-probl}.  
In Section \ref{sec:preliminaries}, we are going to collect some results from
the literature. In Section \ref{sec:airy-potent-bound}, we explain the manner in which we use Airy
potentials for the solution of  elasticity problems, and we state our main
$\Gamma$-convergence result, Theorem \ref{thm:main}.  Section \ref{sec:proof-comp-upper} contains the
proof of the compactness and upper bound part of Theorem \ref{thm:main}, and
Section \ref{sec:proof-lower-bound} contains the proof of the lower bound. The
appendix consists of two parts: In
Section \ref{sec:proof-theor-refthm:r}, we prove some facts on the relaxation of integral functionals whose integrands
depend on second gradients, which we were unable to find in the literature. In
Section \ref{sec:proof-theor-refthm:q}, we derive the 2-quasiconvexification of
the integrand in the compliance minimization  problem.

\subsection*{Notation}
The
symbol ``$C$'' is used as follows: A statement such as ``$f\leq Cg$'' is shorthand for
``there exists a constant $C>0$  such that $f\leq
Cg$''. The value of $C$ may change within the same line. 
For $f\leq Cg$, we also write $f\lesssim g$.


\section{Setting and (a special case of the) main result}
\label{sec:stat-vari-probl}
In the present section, our aim is to present first, the optimal design problems
in linear elasticity, second,  the Michell truss problem  in its variational
form, and third,  a special case of our main theorem that links these problems
via $\Gamma$-convergence. 
On the one hand, this special case does not require a lot of preparatory definitions, and on the
other hand, it is not much weaker than the full result.

For a bounded open set $U\subset\R^n$, $k\in\N$ and $1\leq p\leq\infty$, the Sobolev space $W^{k,p}(U)$ is defined by its norm
\[
\|u\|_{W^{k,p}(U)}=
\sum_{|\alpha|\leq k}\|\partial^\alpha u\|_{L^p(U)}\,,
\]
where the sum runs over multiindices $\alpha=(\alpha_1,\dots,\alpha_n)\in
\N_0^n$ with $|\alpha|=\sum_i\alpha_i\leq k$ and
$\partial^\alpha=\prod_i\partial_i^{\alpha_i}$. For $p=2$, we use the notation
\[
H^k(U)=W^{k,2}(U)\,.
\]
For the spaces with homogeneous boundary conditions, we use the notation
\[
W_0^{k,p}(U)=\{u\in W^{k,p}(U):\nabla^\alpha u=0\text { on }\partial\Omega
\text{ for }|\alpha|\leq k-1\}\,.
\]
The fractional Sobolev spaces $W^{s,p}(U)$ with $k\in\N$, $k<s<k+1$ and $p\in [1,\infty)$  are
defined by the Gagliardo norm,
\[
\|u\|_{W^{s,p}(U)}= \|u\|_{W^{k,p}(U)}+\int_U\d x\int_U\d
y\frac{|\nabla^k u(x)-\nabla^ku(y)|^p}{|x-y|^{n+sp}}\,.
\]
For compact, $n-1$-rectifiable subsets $S\subset \R^n$, we define the norms
$\|u\|_{W^{k,p}(S)}$, $\|u\|_{W^{s,p}(S)}$ by a suitable cover of $S$ by the ranges of Bilipschitz
maps, and we write $W^{s,2}(S)=H^s(S)$. The dual of $H^{s}(S)$ is denoted
by $H^{-s}(S)$, and the dual of $W^{k,\infty}(S)$ is denoted by $W^{-k,1}(S)$.

\medskip

We write $\R^{n\times n}_\sym:=\{M\in\R^{n\times n}:M^T=M\}$.
On
$\R^{n\times n}$, we introduce a scalar product by
\[
\xi:\xi'=\sum_{i,j}\xi_{ij}\xi_{ij}'\,.
\]
We will also use the notation $|\xi|^2=\xi:\xi$ for $\xi\in
\R^{n\times n}$.

\medskip

In the present paper, we are going to derive a result for bounded open sets
$\Omega\subset\R^2$. More precisely, the symbol $\Omega$ will be reserved for
sets satisfying the following definition:

\begin{definition}
  \label{def:H1} From now on, we assume that $\Omega\subset\R^2$ has  the
  following properties:
  \begin{itemize}
  \item[(i)] $\Omega$ is
  open, bounded, connected  and simply connected
\item[(ii)] There exists a finite number of points $x_i\in\partial\Omega$,
  $i=1,\dots,N$, with pairwise disjoint neighborhoods
  $U_i$ of $x_i$ and $C^2$-diffeomorphisms $\varphi_i:U_i\to\varphi(U_i)$ such
  that $\varphi(x_i)=0$,  $\varphi(\Omega\cap U_i)\subset (0,1)^2$ and
  $\varphi(\partial\Omega\cap U_i)\subset[0,1]\times\{0\}\cup\{0\}\times [0,1]$.
\item[(iii)]
  $\partial\Omega$ is $C^2$ away from $\{x_i:i=1,\dots,N\}$.
\end{itemize}
\end{definition}
The purpose of (ii) and (iii) above is that the operator $W^{2,1}(\Omega)\to
L^1(\partial\Omega)$, $u\mapsto \nabla u\cdot n$ is surjective, where $n$
denotes the outer unit normal to $\partial\Omega$, see Theorem \ref{thm:traceop} below.
We will denote the outer unit normal of $\partial\Omega$ by $n=(n_1,n_2)$, and
define a tangent vector $\tau=n^\bot=(-n_2,n_1)$. We denote the tangential
derivative by $\partial_\tau$, and the normal derivative by $\partial_n$.

\subsection{The compliance minimization problem}
\label{sec:compl-minim-probl}

For
$\lambda>0$, let $F_\lambda:\R^{2\times 2}_\sym\to \R$ be defined by
\[
F_\lambda(\xi)=\begin{cases}
0 &\text{ if }\xi=0\\
\lambda+|\xi|^2&\text{ else.}\end{cases}
\]
In the following, the parameter $\lambda$ is a Lagrange multiplier for the
weight of the two-dimensional elastic structure. Taking the limit of vanishing
weight corresponds to the limit $\lambda\to\infty$. 
We define the functionals for finite $\lambda$ with  boundary conditions fixed
by the choice of some
$g\in H^{-1/2}(\partial\Omega;\R^2)$. The space of allowed stresses is given by
\[
S_g(\Omega)=\{\sigma\in L^2(\Omega;\Rsym):\div\sigma=0\text{ in }\Omega,\,\sigma\cdot n=g\text{ on }\partial\Omega\}\,,
\]
where $n$ denotes the unit outer normal of $\Omega$. The integral functional for finite $\lambda$ is given by
\[
\G_{g,\lambda}(\sigma)=\begin{cases}\lambda^{-1/2}\int_{\Omega}F_\lambda(\sigma)\d
  x&
  \text{ if } \sigma\in S_g\\
+\infty & \text{ else.}\end{cases}
\]
For  explanation of the fact that the traditional form of compliance minimization is
equivalent to the minimization of $\G_{g,\lambda}$, see Section \ref{sec:derivation-variational}.

\subsection{The Michell truss problem}
\label{sec:mich-truss-probl}
For the definition of the limit functional, we need to collect some more notation. 

For any Borel set $U\subset\R^n$, 
let
$\M(U)$ denote the set of signed Radon measures on $U$.   We denote
by $\M(U;\R^p)$ the $\R^p$ valued  Radon measures on $U$. Furthermore, let
$\M(U;\R^{n\times n}_\sym)$ denote the space $\{\mu\in\M(U;\R^{n\times
  n}):\mu_{ij}=\mu_{ji}\text{ for }i\neq j\}$. 
For $\mu\in \M(U;\R^{p})$, let $|\mu|$ denote the total variation
measure (see Section \ref{sec:measures}). 
For $\mu\in \M(U;\R^p)$, we have by the Radon-Nikodym differentiation Theorem (see
Theorem \ref{thm:RN} below) that for $|\mu|$-almost every $x\in U$, the derivative
$\d\mu/\d|\mu|$ exists.
For any one-homogeneous function $h:\R^{p}\to\R$ and any $\mu\in
\M(U;\R^{p})$, we may hence define 
\[
h(\mu)=h\left(\frac{\d\mu}{\d|\mu|}\right)\d|\mu|\,.
\]
This is a well defined Borel measure.

\medskip 

Now let $U\subset \R^n$ be open. Let $\mathcal{E}'(U;\R^p)$ denote the dual of $C^1(U;\R^p)$, i.e., the space of
 $\R^p$ valued distributions whose support is compactly contained in $U$.
Let $\mu\in \M(\overline U;\R^n)$,
$f\in \mathcal{E}'(\overline U)$. We say that $-\div\mu =f$ if and only if
\begin{equation}
\int_{\bar
  U}\nabla\varphi\cdot\frac{\d\mu}{\d|\mu|}\d|\mu|=\left<f,\varphi\right>\label{eq:44}
\end{equation}
for every $\varphi\in C^1(\R^n)$. Here $\mu$ and $f$ are viewed, respectively,
as a measure and a
distribution  
on $\R^n$ supported on $\bar U$. If $f$ has support in the boundary $\partial
U$, then this induces a boundary condition for $\mu$. Just as in the equation
above, the notation
$\left<\cdot,\cdot\right>$ will denote the pairing of topological vector spaces with
their dual in the sequel. It will always be clear from the context which pairing is meant. For
$\mu\in\M(U;\R^{n\times n}_{\mathrm{sym}})$ and $f\in \mathcal{E}'(U;\R^n)$, we say that $-\div\mu=f$ if the
equation holds for every row, $-\div\mu_i=f_i$ for $i=1,\dots,n$.

\medskip

For $\xi\in\Rsym$, let $\lambda_1(\xi), \lambda_2(\xi)$ denote the eigenvalues
of $\xi$. 
We set
\[
\rho^0(\xi):=\sum_{i=1}^2 |\lambda_i(\xi)|\,.
\]
We will repeatedly use the following estimates:
\begin{equation}
|\xi|\leq \rho^0(\xi)\leq 2|\xi|\,.\label{eq:30}
\end{equation}
Note that $\rho^0:\Rsym\to\R$ is sublinear and positively one-homogeneous.

For $U\subset\R^2$, $\mu=(\mu_1,\mu_2)\in \M(U;\R^2)$, we write $\mu^\bot:=(-\mu_2,\mu_1)$ and
$\curl\mu=\div \mu^\bot$. Again, for
$\mu\in\M(U;\Rsym)$ and $f\in \mathcal{E}'(U;\R^2)$, we say that $\curl\mu=f$ if the
equation holds for every row.

Now let $\Omega$ be as in Definition \ref{def:H1}. For $g\in W^{-1,1}(\partial\Omega;\R^2)$, let the space of permissible stresses be given by
\[
\Sigma_g(\Omega)=\{\sigma\in \M(\overline \Omega;\Rsym):-\div \sigma=g\H^1\ecke\partial\Omega\}\,.
\]


With these preparations, we are ready to define the Michell problem for  for traction boundary values $g\in W^{-1,1}(\partial\Omega;\R^2)$,
\[
\G_{g}(\sigma)=\begin{cases} 2\rho^0(\sigma)(\overline{\Omega})& \text{ if } \sigma\in\Sigma_g(\overline{\Omega})\\
+\infty & \text{ else.}\end{cases}
\]
For a motivation of this functional in the context of structural optimization, we refer to  \cite{MR1859696,allaire1993optimal}.

\newcommand{\ul}{g}
\subsection{Gamma convergence}
\label{sec:gamma-convergence}
We want to approximate the functional $\G_{g}$ by the functionals
$\G_{\ul,\lambda}$ in the sense of $\Gamma$-convergence. We assume that $
\Omega$ satisfies Definition \ref{def:H1}. We  introduce the trace operators
\[
\begin{split}
  \gamma_0:&u\mapsto u|_{\partial\Omega}\\
  \gamma_1:&u\mapsto \nabla u|_{\partial\Omega}\cdot n\,.
\end{split}
\]
For the properties of the trace operators, see Section \ref{sec:trace-operators}
below. 


As a special case of our main theorem, we have that for $g\in H^{-1/2}(\partial\Omega;\R^2)$, 
\[
\G_{g,\lambda}\stackrel{\Gamma}{\to} \G_{g}\,.
\] 
More precisely:
\begin{theorem}
\label{thm:special}
 Let $\Omega\subset\R^2$ satisfy Definition \ref{def:H1} and let  $g\in H^{-1/2}(\partial\Omega;\R^2)$.
\label{thm:gammabdry}
  \begin{itemize}
  \item[(i)] \emph{Compactness:} Let $\{\sigma_\lambda\}_\lambda\subset L^2(\Omega;\Rsym)$ be
  such that $\G_{g,\lambda}(\sigma_\lambda)<C$. Then there exists a subsequence (no
  relabeling) and $\sigma\in \M(\Omega;\Rsym)$ such that $\sigma_\lambda\to \sigma$ weakly * in
  the sense of measures. 
\item[(ii)] \emph{Lower bound:} If $\sigma_\lambda\to \sigma$ weakly * in the sense of
  measures, then
  \begin{equation}
  \liminf_{\lambda\to\infty}\G_{g,\lambda}(\sigma_\lambda)\geq \G_{g}(\sigma)\,.\label{eq:40}
  \end{equation}

\item[(iii)] \emph{Upper bound:} For every $\sigma\in\M(\Omega;\Rsym)$ there exists a sequence
  $\{\sigma_\lambda\}_\lambda\subset L^2(\Omega;\Rsym)$ such that $\sigma_\lambda\to \sigma$
  weakly * in the sense of measures and $\lim_{\lambda\to\infty}\G_{g,\lambda}(\sigma_\lambda)=\G_{g}(\sigma)$.
  \end{itemize}
\end{theorem}

\begin{remark}
\label{rem:mainrem}
  \begin{itemize}
\item[(i)] For the sake of simplicity, we have here  set the same boundary conditions $g\in
  H^{-1/2}(\partial\Omega;\R^2)$ for the approximating and the limit
  problem. Actually,  one would like to obtain a larger class of allowed 
  ``boundary conditions'' for the limit problem. For example, one would like to consider
  $g=\sum_{i=1}^Mv_i \delta_{x_i}$, where $x_i\in \partial\Omega$, $v_i\in \R^2$
  for $i=1,\dots,M$ and $\delta_x$ denotes the distribution defined by
  $\delta_x(f)=f(x)$. These boundary values are the ones that one considers in
  the Michell problem, see \cite{bouchitte2008michell}. Distributions $g$ of this
  type are  not in $H^{-1/2}(\partial\Omega;\R^2)$, but they do belong to
  $W^{-1,1}(\partial\Omega;\R^2)$, which at first glance might look like the
  ``natural'' space for the boundary conditions of the limit problem. In our
  main theorem, we will allow  boundary values from  a certain subset of
  $W^{-1,1}(\partial\Omega;\R^2) $ in the limit problem. In particular, this
  subset contains the aforementioned ``delta-type'' distributions, provided that
  the applied forces do not act tangentially, see Lemma \ref{lem:deltabdry}. The functions
  in this space will be
  approximated by boundary values $g_\lambda\in H^{-1/2}(\partial\Omega;\R^2)$.
  \item[(ii)] The main idea of our proof can be summarized as follows: By the
    well-known representation of divergence free stresses via Airy potentials,
    we can formulate the variational problems for finite and vanishing weight as
    the minimization of integral functionals in the spaces $H^{2}$ and $BH$ respectively, where
    the latter denotes the space of functions of bounded Hessian, i.e., the
    space of functions $u\in W^{1,1}$ such that the distributional derivative
    $D^2u$ defines a vector-valued Radon measure. We may then use the blow-up
    technique developed by Fonseca and M\"uller
    \cite{MR1718306,MR1218685,ambrosio1992relaxation} and the results by Kohn
    and Strang \cite{MR820342} and Allaire and Kohn \cite{allaire1993optimal} to
    prove the lower bound part of the $\Gamma$-convergence result. For the
    construction of the upper bound, we use  approximation and relaxation
    results that are well known to specialists. Nevertheless, for some of them,
    we could not find a proof in the literature and provide them here.
\item[(iii)] In their formal derivations of the Michell truss problem in
  \cite{allaire1993optimal}, Allaire and Kohn  also discussed the
  three-dimensional case. We are not able to say anything new about this case: The representation of divergence free stresses via Airy potentials
  is limited to  two dimensions.

\item[(iv)]
  We restrict ourselves to the case of vanishing Poisson ratio for the sake of
  simplicity and readability. The interested reader will be able to generalize
  our results without difficulty to a general ``soft'' isotropic phase, defined by the elasticity
  tensor $A_0\in \mathrm{Lin}(\Rsym;\Rsym)$ with
\[
A_0\xi= 2\mu \left(\xi-\frac12 (\Tr \xi) \id_{2\times 2}\right)+\kappa \Tr \xi
\id_{2\times 2}\quad\text{ for }\xi\in\Rsym\,,
\]
where $\mu,\kappa$ are the shear and bulk modulus respectively, and 
$\mathrm{Lin}(V;W)$ denotes the set of linear operators $V\to W$. In that case,
the functionals for finite $\lambda$ are given by 
\[
\G^{A_0}_{g,\lambda}(\sigma)=\begin{cases}\lambda^{-1/2}\int
  F^{A_0}_\lambda(\sigma)\d x& \text{ if } \sigma\in
  S_g\\+\infty&\text{ else, }\end{cases}
\]
where 
\begin{equation}
F^{A_0}_\lambda(\xi)=\begin{cases}0&\text{ if }\xi=0\\
(A_0^{-1}\xi):\xi+\lambda&\text{ else,}\end{cases}\label{eq:49}
\end{equation}
and the limit functional is given by
$\frac{\kappa+\mu}{\sqrt{4\kappa\mu}}\G_g$. It suffices to take the formulas for
the quasiconvex envelope for $F^{A_0}_\lambda$ from \cite{allaire1993optimal}, and
adapt our proof accordingly.
\end{itemize}
\end{remark}
\subsection{Derivation of the variational  form of the compliance minimization problem}
\label{sec:derivation-variational}
We give a brief derivation of the compliance minimization problem in its
variational form, 
\[
\inf_{\sigma\in S_g(\Omega)}\G_{g,\lambda}(\sigma)\,,
\]
starting from the standard formulation in linear elasticity. What we present
here is a subset of the derivation by Allaire and Kohn, see \cite{MR1859696} for
more details.

\medskip

As before, let $g\in H^{-1/2}(\partial \Omega;
\R^2)$.
Consider an elastic body $\Omega\subset\R^2$, characterized by its
elasticity tensor $A_0\in  \mathrm{Lin}(\Rsym;\Rsym)$, where we assume that
$A_0$ is invertible.
  We remove a subset $H\subset \Omega$ from the elastic body and the new boundaries from that
process shall be traction-free. The resulting linear elasticity problem is to
find $u:\Omega\setminus H\to\R^2$ such that
\[
\begin{split}
  \sigma&=A_0 e(u)\\
  \div \sigma&=0\quad\text{in }\Omega \\
  \sigma\cdot n&=g\quad \text{on }\partial \Omega\\
 \sigma\cdot n&=0\quad \text{on
  }\partial H\,,
\end{split}
\]
where $e(u)=\frac12(\nabla u+\nabla u)^T$.
The compliance (work done by the load) is given by 
\[
c(H)=\int_{\partial\Omega}g\cdot u\d\H^1=\int_{\Omega\setminus H}(A_0
e(u)):e(u)\d x\,,
\]
where $u:\Omega\setminus H\to \R^2$ is the unique solution to the linear elasticity system above.
We want to minimize the compliance under a constraint on the ``weight'' $\L^2(\Omega\setminus H)$.
We do so by the introduction of a Lagrange multiplier $\lambda$, and are
interested in the minimization problem
\[
\min_H \left(c(H)+\lambda \L^2(\Omega\setminus H)\right)\,.
\]
Taking the limit of vanishing
weight corresponds to the limit $\lambda\to\infty$. 
We rephrase the problem by  considering space-dependent elasticity  tensors 
of the form $A(x)=\chi(x)A_0$, where $\chi\in L^\infty(\Omega;\{0,1\})$.
The elasticity system from above becomes

\begin{equation}
\begin{split}
  \sigma=&A(x) e(u)\\
  \div \sigma=&0\quad\text{in }\Omega \\
  \sigma\cdot n=&g\quad \text{on }\partial \Omega\,.
\end{split}\label{eq:55}
\end{equation}
Now the compliance is a functional on the set of permissible elasticity tensors,
and is given by 
\[
c(A)=\int_{\Omega}(A(x)
e(u)):e(u)\d x\,,
\]
where $u$ is the solution of \eqref{eq:55}.
By the principle of minimum complementary energy, we have that the compliance
can be written as
\[
c(A)=\int_\Omega G(A(x),\sigma(x))\d x\,,
\]
where
\[
G(\bar A,\xi)=\begin{cases} +\infty & \text{ if }\xi\neq 0 \text{ and }\bar A=0\\
0 & \text{ if }\xi= 0 \text{ and }\bar A=0\\
(\bar A^{-1}\xi):\xi&\text{ else, } \end{cases}
\]
and $\sigma\in L^\infty(\Omega;\Rsym)$ is a solution of the PDE
\[
\begin{split}
  \div\sigma&=0\quad\text{ in }\Omega \\
\sigma\cdot n&=g\quad\text{ on
  }\partial\Omega\,,
\end{split}
\]
i.e., $\sigma\in S_g(\Omega)$.
We see that the compliance minimization problem can be understood as the
variational problem of finding the infimum 
\[
\inf\left\{ \int_\Omega \left(G(\chi(x) A_0,\sigma(x))+\lambda \chi(x)\right)\d x:
  \chi\in L^\infty(\Omega;\{0,1\}),\,\sigma\in S_g(\Omega)\right\}\,.
\]
Of course, the compliance of a pair $(\chi,\sigma)$ is infinite if  there exists
a set of 
positive measure $U$ such that $\chi=0$ and $\sigma\neq 0$ on $U$. Hence the
above variational problem is equivalent with 
\begin{equation}
\inf\left\{ \int_\Omega F^{A_0}_\lambda(\sigma)\d x:\sigma\in S_g(\Omega)
  \right\}\,,\label{eq:57}
\end{equation}
where $F_\lambda^{A_0}$ has been defined in \eqref{eq:49}.
As is well known, the variational problem \eqref{eq:57} does not possess a
solution in general and requires relaxation. For simplicity, we assume here that
$A_0$ is the identity on $\Rsym$, see Remark  \ref{rem:mainrem}. Then  \eqref{eq:57} is just the variational problem $\G_{g,\lambda}(\sigma)\to \inf$.

\section{Preliminaries}
\label{sec:preliminaries}

For the proof of our main result, we are going to rely heavily on two sets of
results from the literature: On the one hand,
on the results on optimal design in the ``relaxed formulation''
\cite{MR820342,allaire1993optimal,MR1859696}, and on the other hand,  on the blow-up technique for the derivation of relaxed functionals and
$\Gamma$-limits developed by Fonseca and M\"uller
\cite{MR1177778,ambrosio1992relaxation,MR1218685}. In order to present them, we need to
review some basic facts about measures, $BV$ functions and quasiconvexity.

In the following, let $U\subset\R^n$ be open and bounded.

\subsection{Measures}
\label{sec:measures}

At the basis of the blow-up argument that we use in the present work is a
refinement of the well known Radon-Nikodym
differentiation theorem:
\begin{theorem}[Proposition 2.2 in \cite{ambrosio1992relaxation}]
\label{thm:RN}
  Let $\lambda, \mu$ be Radon measures in $U$ with $\mu\geq 0$. Then there
  exists a Borel set $E\subset U$ with $\mu(E)=0$ such that for any $x_0\in  \supp \mu\setminus E$ we have
\[
\lim_{\rho\downarrow 0} \frac{\lambda(x_0+\rho K)}{\mu(x_0+\rho
  K)}=\frac{\d\lambda}{\d\mu}(x_0)
\]
for any bounded convex set $K$ containing the origin. Here, the set $E$ is
independent of $K$.
\end{theorem}

Let $\mu_j\in \M(U;\R^p)$ for $j=1,2,\dots$. We say that $\mu_j\to \mu\in \M(U;\R^p)$
weakly * in the sense of measures if
\[
\int_{U}\varphi\cdot\frac{\d\mu_j}{\d|\mu_j|}\d|\mu_j|\to 
\int_{U}\varphi\cdot\frac{\d\mu}{\d|\mu|}\d|\mu|
\quad\text{ for every } \varphi\in C^0_c(U;\R^p)\,.
\]
It is a well known fact that if $\sup_{j\in\N}|\mu_j|(U)<\infty$, then
there exists a weakly * convergent subsequence.

The next result concerns the convergence of positively
one-homogeneous functions of measures. The statement below is contained in Theorem 1.15 in \cite{demengel1990weak}.
\begin{theorem}
\label{thm:homomeas}
Let $h:\R^p\to\R$ be positively one-homogeneous and continuous, and let $\mu_j\in
\M(U;\R^p)$, $j\in\N$, such that $\mu_j\to \mu$, $|\mu_j|\to |\mu|$ weakly
* in the sense of measures. Then
\[
h(\mu_j)\to h(\mu)\quad\text{ weakly * in the sense of measures.}
\]
\end{theorem}

\subsection{$BV$ and $BH$ functions}
\label{sec:bv-functions}
The space of functions of bounded variation
$BV(U)$ is defined as the set of functions $f\in L^1(U)$ that satisfy
\[
\sup\left\{-\int_{U} f\div \varphi\,\d x:\varphi\in
C^1_c(U;\R^n),\, \|\varphi\|_{C^0}\leq 1\right\}<\infty\,.
\]
In this case, the map $\varphi\mapsto -\int_{U} f\div \varphi\,\d x$ defines a
vector valued Radon measure, which is denoted by $Dv$.

According to the Theorem~\ref{thm:RN}, we have the following decomposition of the
measure $Dv$ for $v\in BV(U)$,
\[
Dv=\nabla v \,\L^n+D^sv\,,
\]
where $\nabla v=\frac{\d(Dv)}{\d\L^n}$ and $D^sv$ is the so-called singular part
of $Dv$.

The following theorem determines the structure of  blow-ups of $BV$ functions:

\begin{theorem}[Theorem 2.3 in \cite{ambrosio1992relaxation}]
  \label{thm:BVblow}
Let $u\in BV(U;\R^m)$ and for a bounded convex open set $K$ containing the
origin, and let $\xi$ be the density of $Du$ with respect to $|Du|$,
$\xi=\frac{\d(Du)}{\d(|Du|)}$.
For $x_0\in \supp (|Du|)$, assume that $\xi(x_0)=\eta\otimes \nu$ with $\eta\in\R^m$,
$\nu\in\R^n$, $|\eta|=|\nu|=1$, and for
$\rho>0$ let 
\[
v_\rho(y)=\frac{\rho^{n-1}}{|Du|(x_0+\rho K))}\left(u(x_0+\rho
y)-\fint_{x_0+\rho K}u(x')\d x'\right)\,.
\]
Then for every $\sigma\in (0,1)$ there exists a sequence $\rho_j$ converging to
0 such that $v_{\rho_j}$ converges in $L^1(K;\R^m)$ to a function $v\in
BV(K;\R^m)$ which satisfies $|Dv|(\sigma \overline{K})\geq \sigma^n$ and can be
represented as 
\[
v(y)=\psi(y\cdot \nu)\eta
\]
for a suitable non-decreasing function $\psi:(a,b)\to \R$, where $a=\inf\{y\cdot
\nu:y\in K\}$ and $b=\sup\{y\cdot
\nu:y\in K\}$.
\end{theorem}

We will also need Alberti's  rank-one theorem:

\begin{theorem}
\label{thm:alberti}
 Let $v\in BV(U;\R^m)$. Then $D^sv$  is rank-one.
\end{theorem}


For later reference,  we also mention that for  $u\in W^{1,1}(U;\R^m)$, we have as a
consequence of the classical Sobolev embeddings, that for almost every $x_0\in
U$, we have
\begin{equation}
\lim_{\e\to 0}\frac{1}{\e}\left(\fint_{Q(x_0,\e)}|u(x)-u(x_0)-\nabla u(x_0)\cdot(x-x_0)|^{n/(n-1)}\d
  x\right)^{(n-1)/n}=0\,.\label{eq:47}
\end{equation}
Here, we have used the notation $Q(x_0,\e)=x_0+[-\e/2,\e/2]^n$.

\medskip

The space of functions of bounded Hessian is defined as
\[
BH(U)=\{u\in W^{1,1}(U), \nabla u\in BV(U;\R^n)\}\,.
\]
It can be made into a normed space by setting
\[
\|u\|_{BH(U)}=\|u\|_{W^{1,1}(U)}+|D^2u|(U)\,.
\]
We say that a sequence $u_j\in BH(U)$ converges weakly * to $u\in
BH(U)$ if $u_j\to u$ in $W^{1,1}(U)$ and $D^2u_j\to D^2 u$ weakly * in
$\M(U;\R^{n\times n})$.

The space $BH(U)$ has been investigated first in
\cite{demengel1984fonctions,demengel1989compactness}. In particular, these
papers contain theorems about compactness and extension properties of this
space. 
The first theorem that we cite is a weakened form of Theorem 1.3 in \cite{demengel1989compactness}:
\begin{theorem}
\label{thm:BHcompact}
Let $u_j$ be a bounded
  sequence in $BH(U)$. Then there exists a subsequence (no relabeling) and
  $u\in BH(U)$ such that
\[
u_j\to u\quad\text{ weakly * in }BH(U)\,.
\]
\end{theorem}

In two dimensions, functions in $BH$ are continuous:
\begin{theorem}[\cite{demengel1984fonctions}, Theorem 3.3]
\label{thm:BHcont}
Let $U\subset\R^2$ with $C^2$ boundary. Then
\[
BH(U)\subset C^0(U)\,.
\]
\end{theorem}

\subsection{Relaxation of integral functionals that depend on higher
  derivatives}
In the proof of the upper bound of Proposition \ref{prop:gammaairy}, we will need a relaxation result for  integral functionals that depend on second gradients.
This is the case $k=p=2$ in Theorem \ref{thm:relax2} below. 

\medskip

A
function $f:\R^{m\times n^k}\to \R$ is called $k$-quasiconvex if

\begin{equation}
f(\xi)=\inf\left\{\int_{[-1/2,1/2]^n}f(\xi+\nabla^k \varphi)\d x:\varphi\in
  W^{k,\infty}_0([-1/2,1/2]^n;\R^m)\right\}\,,\label{eq:23}
\end{equation}
see \cite{MR0188838}.  
The so-called  $k$-quasiconvexification of $f:\R^{m\times n^k}\to \R$ is given by
the right hand side above,
\[
Q_kf(\xi)=\inf\left\{ \int_{[-1/2,1/2]^n}
f(\xi+\nabla^k\varphi)\d x:\,\varphi\in W^{k,\infty}_0([-1/2,1/2]^n;\R^m)\right\}\,.
\]
As is well known, in the case $k=1$, one obtains the relaxation of integral
functionals $u\mapsto\int f(\nabla u)\d x$ by replacing $f$ by its
1-quasiconvexification $Q_1f$. Concerning higher $k$, there exist some relaxation results in the literature, but we could not find any
theorems that fit our situation. In particular, Theorem 1.3 in
\cite{braides2000quasiconvexity} deals with the case of the relaxation of
Caratheodory functions. In our case, where the integrand only depends on the
second gradient, this means that continuity of the integrand  $f(\xi)$
with respect to $\xi$ is required. We are interested in the non-continuous case
$f=F_\lambda$, so we cannot use this theorem. The following theorem suits our
purpose, and we prove it in the appendix:

\begin{theorem}
\label{thm:relax2}
  Let $1\leq p<\infty$, and let $f:\R^{m\times n^2}\to\R$ such that
\[
0\leq f(A)\leq C (1+|A|^p)\quad\text{ for all } A\in \R^{m\times n^2}\,.
\]
Let $\Omega\subset \R^2$ satisfy Definition \ref{def:H1}, and let
$u_0\in W^{2,p}(\Omega)$. Let $u\in u_0+W^{2,p}_0(\Omega;\R^m)$ and $\e>0$. Then there
exists $v\in u_0+W^{2,p}_0(\Omega;\R^m)$ with
\[
\begin{split}
  \|u-v\|_{W^{1,p}(\Omega;\R^m)}&<\e\\
  \int_\Omega f(\nabla^2 v)\d x&< \int_\Omega Q_2f(\nabla^2 u)\d x+\e\,.
\end{split}
\]
\end{theorem}
\begin{remark}
The theorem can be straightforwardly generalized to the case of higher
derivatives, if in Definition \ref{def:H1} one
replaces $C^2$-regularity with $C^k$-regularity of $\partial \Omega$ in the appropriate
sense. Another straightforward generalization is the case of general dimension
$n$ of the domain $\Omega$. Moreover, (simple)  connectedness and boundedness of
$\Omega$ are  not
necessary here.
\end{remark}

We will need to
determine  the 2-quasiconvexification of $F_\lambda:\R^{2\times 2}\to \R$. In
principle this is contained in \cite{MR820342,allaire1993optimal}. However we could not
find a clear statement  in the literature, so we give a proof of the
following theorem in the appendix.
\begin{theorem}
\label{thm:Q2F}
  We have
\[
Q_2F_\lambda(\sigma)=\begin{cases}2\sqrt{\lambda} \rho^0(\sigma)-2|\det \sigma| & \text{ if }\rho^0(\sigma)\leq
  \sqrt{\lambda}\\
 |\sigma|^2+\lambda & \text{ else.}\end{cases}
\]
\end{theorem}

\subsection{Trace and extension operators}
\label{sec:trace-operators}
Recall that we assume that  $\Omega\subset\R^2$ satisfies Definition \ref{def:H1}. The trace operator
\[
\gamma_0:u\mapsto u|_{\partial\Omega}
\]
is linear surjective as a map $W^{1,1}(\Omega)\to L^1(\Omega)$ (see
\cite{gagliardo1957caratterizzazioni}) and also as a map $BV(\Omega)\to
L^1(\Omega)$ (see \cite{miranda1967comportamento}). For the spaces
$W^{2,1}(\Omega)$ and $BH(\Omega)$, it also makes sense to consider the operator
\[
\gamma_1:u\mapsto  \nabla u|_{\partial\Omega}\cdot n\,.
\]
The following theorem combines statements from  Chapter 1.8 of
\cite{lions2012non} and  Chapter 2 as well as the appendix of \cite{demengel1984fonctions}.
\begin{theorem}
  \label{thm:traceop}
  \begin{itemize}
  \item[(i)] The operator 
$(\gamma_0,\gamma_1)$ is linear surjective  both as a map 
\[
H^{2}(\Omega)\to H^{3/2}(\partial\Omega)\times H^{1/2}(\partial\Omega)
\]
and as a map
\[
BH(\Omega)\to \gamma_0(W^{2,1}(\Omega))\times L^1(\partial\Omega)\,.
\]
\item[(ii)] There exist  continuous right inverses $(\gamma_0,\gamma_1)^{-1}$,
  defined as maps
\[
H^{3/2}(\partial\Omega)\times H^{1/2}(\partial\Omega)\to H^{2}(\Omega)
\]
and 
\[
\gamma_0(W^{2,1}(\Omega))\times L^1(\partial\Omega)\to W^{2,1}(\Omega)\,.
\]
  \end{itemize}

\end{theorem}

\begin{remark}
  \begin{itemize}
  \item[(i)] The norm on $\gamma_0(W^{2,1}(\Omega))$ is the one induced by
$\gamma_0$,
\[
\|u\|_{\gamma_0(W^{2,1}(\Omega))}:=\inf\{\|v\|_{W^{2,1}(\Omega)}:\gamma_0
(v)=u\}\,.
\]
Note that $\gamma_0(W^{2,1}(\Omega))\subsetneq W^{1,1}(\partial\Omega)$, see the
appendix of \cite{demengel1984fonctions}. This fact together with the theorem
explains our choice of assumptions on the boundary conditions
in our main theorem; see equations \eqref{eq:25} and \eqref{eq:42} below. 
\item[(ii)] In \cite{demengel1984fonctions}, the statement on  the surjectivity of the trace operator $BH(\Omega)\to
  \gamma_0(W^{2,1}(\Omega))\times L^1(\partial\Omega)$ 
   is only made  for $C^2$-regular boundary. For the sake of brevity, we only
   sketch the changes of that proof  that have to be made to show that
   the claim also holds true for  $\Omega\subset\R^2$
   satisfying Definition \ref{def:H1}. In Proposition 1 of the appendix of
   \cite{demengel1984fonctions}, it is shown that for 
$g\in L^1([0,1]\times \{0\})$, there
   exists $u\in W^{2,1}([0,1]^2)$ such that $u|_{[0,1]\times \{0\}}=0$,
   $\partial_2 u|_{[0,1]\times \{0\}}=g$. The proof works by an explicit
   construction, modifying an idea by Gagliardo. In fact, using the explicit
   formulas for $u$ and its partial derivatives that are given in the proof, one easily deduces that $u_{\{0\}\times[0,1]}=\partial_1
   u|_{\{0\}\times[0,1]}=0$. Hence, one may use the Proposition twice to obtain
   $\tilde u\in W^{2,1}([0,1]^2)$ that vanishes on $\Gamma:=[0,1]\times\{0\}\cup
   \{0\}\times [0,1]$, and whose normal derivative agrees with a given $\tilde
   g\in L^1(\Gamma)$ on $\Gamma$. With this slightly more general version of the
   Proposition, the proof of Theorem 1 in the appendix of
   \cite{demengel1984fonctions}, which states the surjectivity of the trace operator, can be extended to the case where $\Omega$
   satisfies  Definition \ref{def:H1} without additional changes: One
   uses a suitable cover of the boundary $\partial\Omega$ by open sets, an associated partition
   of unity and $C^2$-regular diffeomorphisms that reduce the problem to the
   situation of the proposition.
  \end{itemize}
\end{remark}

\medskip

We have the following Poincar\'e inequality for $BH$:
\begin{lemma}
\label{lem:BHpoincare}  Let $u\in BH(\Omega)$, $\gamma_0(u)=f$, $\gamma_1(u)=g$, then we have
\[
\|u\|_{W^{1,1}(\Omega)}\lesssim
\|f\|_{W^{1,1}(\partial\Omega)}+\|g\|_{L^1(\partial\Omega)}+|D^2 u|(\Omega)\,.
\]
\end{lemma}
\begin{proof}
  By the Poincar\'e inequality for $BV$ functions (see \cite{MR1014685} Chapter
  5) we have
\[
\|\nabla u\|_{L^1(\Omega)}\lesssim \|\nabla u\|_{L^1(\partial\Omega)}+|D^2
u|(\Omega)\,
\]
and 
\[
\|u\|_{L^1(\Omega)}\lesssim \|u\|_{L^1(\partial\Omega)}+\|\nabla
u\|_{L^1(\Omega)}\,.\]
This proves the claim.
\end{proof}



\medskip

Next, we 
quote a general extension result from \cite{stein2016singular}.  In the following, we slightly change a definition from
\cite{stein2016singular}:
\begin{definition}
\label{def:minimal}
Let $U\subset\R^n$ be open and
bounded. We
say that the boundary $\partial U$ is said to satisfy minimal conditions if 
\begin{itemize}
\item[(i)] There exists a cover of $\partial U$ by a finite  number of
  open sets $U_1,U_2,\dots,U_m$
\item[(ii)] For every $i=1,\dots,M$, $\partial U\cap U_i$ can be represented as
  the graph of a Lipschitz function $\tilde U_i\to \R$ with $\tilde
  U_i\subset\R^{n-1}$. 
\end{itemize}
  \end{definition}
Note that if $ \Omega$ satisfies Definition \ref{def:H1}, then $\partial\Omega$ also satisfies the
minimal conditions.
\begin{theorem}[Theorem 5 and 5' in Chapter 6 of \cite{stein2016singular}]
  \label{thm:steinext}
Let $U\subset\R^n$ such that $\partial U$  satisfies the minimal conditions. Then there exists an extension operator
\[
E:L^1(U)\to L^1(\R^n)
\]
that is continuous as a map $W^{k,p}(U)\to W^{k,p}(\R^n)$ for every
$k\in\N$ and every $1\leq p\leq \infty$. Moreover, the norm of this operator
only depends on $n$ and on the maximum of the Lipschitz constants of the
functions  that appear   in Definition \ref{def:minimal} (ii).
\end{theorem}



\section{Airy potentials and boundary conditions}
\label{sec:airy-potent-bound}
In the present section, we assume that $\Omega\subset\R^2$ satisfies Definition \ref{def:H1}.
\subsection{Airy potentials}
Here we are going to rephrase the compliance minimization problem and the Michell problem.
We use the  representation of divergence free stresses in two
dimensions by Airy potentials. We recall that for $A\in\Rsym$, the
cofactor matrix $\cof A$ is defined by
\[
\cof A=\left(\begin{array}{cc} A_{22} & -A_{12}\\-A_{12} &
    A_{11}\end{array}\right)\,.
\]
Note that in two dimensions, we have $\cof\cof A=A$.

In the compliance minimization problem, we say that $u\in H^{2}(\Omega)$
is an Airy potential  for $\sigma\in S_g(\Omega)$ if

\begin{equation}
\nabla^2 u=\cof \sigma \quad\text{ in }\Omega\,.\label{eq:22}
\end{equation}
Note that in such a situation, we have $\div\sigma=\curl\cof\sigma=0$.
Since $A\mapsto \cof A$ is linear on two by two matrices, the object $\cof \mu$
is well defined for $\mu\in\M(\overline \Omega;\Rsym)$. 
We say that the function $u\in W^{1,1}(U)$ is
an Airy potential for
$\sigma\in \Sigma_g(\Omega)$ if $U$ is a neighborhood of $\overline\Omega$,
and
\begin{equation}
 D^2 u\ecke\overline{\Omega}= \cof \sigma \label{eq:45}
 \end{equation}
as elements of $\M(\overline\Omega;\Rsym)$.
Our definitions of  Airy potentials make sense by the Poincar\'e Lemma;
this statement is made precise in the following lemma. 
\begin{lemma}
\label{lem:airylem}
  We have 
\[
\{\sigma\in\M(\Omega;\Rsym):\curl\sigma=0\}=\{D^2u:u\in BH(\Omega)\}
\]
and
\[
\{\sigma\in L^2(\Omega;\Rsym):\curl\sigma=0\}=\{\nabla^2u:u\in H^2(\Omega)\}\,.
\]
\end{lemma}
\begin{proof}
The inclusion $\{D^2 u:u\in BH(\Omega)\}\subseteq \{\sigma\in\M(\Omega;\Rsym):\curl\sigma=0\}$ is
obvious. For the opposite inclusion, let $\sigma\in \M(\Omega;\Rsym)$ with
$\curl \sigma=0$. Let $\varphi$ be a standard mollifier,
i.e., $\varphi\in C_c^\infty(\R^2;\R^2)$, $\supp\varphi\subset \{x\in\R^2:|x|<1\}$,
$\int_{\R^2}\varphi\d x=1$, and
$\varphi_\e:=\e^{-2}\varphi(\cdot/\e)$.  On
${\bar \Omega}_\e:=\{x\in \Omega:\dist(\partial \Omega,x)>\e\}$, set
$\sigma_\e:=\sigma*\varphi_\e$. Note that we have $\sigma_\e\in
C^\infty({\bar \Omega}_\e;\Rsym)$ with $\curl\sigma_\e=0$ on ${\bar \Omega}_\e$. For $\e$ small
enough, ${\bar \Omega}_\e$ is simply connected, and by the Poincar\'e Lemma there
exists $v_\e\in C^\infty({\bar \Omega}_\e;\R^2)$ such that $\nabla v_\e=\sigma_\e$. For
every $(\bar x,\bar y)\in {\bar \Omega}_\e$, there exists an open square $Q\subset
{\bar \Omega}_\e$ with center $(\bar x,\bar y)$. On $Q$, we have
\[
v_\e( x, y)=v(\bar x,\bar y)+\left( \begin{array}{c}\int_{\bar x}^x (\sigma_\e)_{11}(t,0)\d
  t+\int_{\bar y}^y(\sigma_\e)_{12}(x,t)\d t\\
\int_{\bar x}^x (\sigma_\e)_{21}(t,y)\d
  t+\int_{\bar y}^y(\sigma_\e)_{22}(0,t)\d t\end{array}\right)\,.
\]
Using $(\sigma_\e)_{12}=(\sigma_\e)_{21}$, one easily obtains $\curl v_\e=0$ on
$Q$, and hence on all of ${\bar \Omega}_\e$.
Again by the Poincar\'e Lemma there exists $\tilde u_\e\in
C^\infty({\bar \Omega}_\e)$ such that $\sigma_\e=\nabla^2 \tilde u_\e$ on ${\bar
  \Omega}_\e$. 
Of course, the sets $\bar\Omega_\e$ have Lipschitz boundary. Moreover, there
exist open sets $U_1,\dots,U_M$ such that for $\e$ small
enough,  $U_i\cap\partial\bar
\Omega_\e$ can be represented as the graph of some Lipschitz function $w_{i,\e}$, and the Lipschitz
constants of $w_{i,\e}$ are uniformly bounded. By Theorem \ref{thm:steinext},
we may
 extend $\tilde u_\e\in W^{2,1}(\bar\Omega_\e)$ to $u_\e\in W^{2,1}(\Omega)$ such that 
\begin{equation}
\|u_\e\|_{W^{2,1}(\Omega)}\leq C\|\tilde u_\e\|_{W^{2,1}({\bar \Omega}_\e)}\,,\label{eq:24}
\end{equation}
where $C$ does not depend on $\e$.
After subtracting
suitable affine functions, we may assume
\[
\int_\Omega u_\e\d x=0,\quad\int_\Omega \nabla u_\e\d x=0\,.
\]
From \eqref{eq:24} and the Poincar\'e inequality in $BH$ (see \cite{demengel1984fonctions}) it follows
\[
\|u_\e\|_{BH(\Omega)}\lesssim \|\nabla^2 u_\e\|_{L^{1}(\Omega)}\lesssim |\sigma|(\Omega)\,.
\]
By Theorem \ref{thm:BHcompact}, we  obtain that 
there exists $u\in BH(\Omega)$ such that $u_\e\to u$ in $BH(\Omega)$, with
$D^2u=\sigma$. This proves the first statement. The second statement is proved
in the same way, using Theorem \ref{thm:steinext} for the extension
$H^2(\bar\Omega_\e)\to H^2(\Omega)$, and weak compactness of the resulting bounded
sequence $u_\e$ in $H^2(\Omega)$.
\end{proof}

\subsection{Boundary values}

We say that $g\in W^{-1,1}(\partial\Omega;\R^2)$ is balanced if 
\[
\int_{\partial\Omega}(M x+b)\cdot g(x)\d\H^1=0\quad\text{ for all }M\in
\R^{2\times 2}_{\mathrm{skew}}\text{ and }b\in\R^2\,.
\]
It only makes sense to consider balanced traction boundary values, as can be
seen from the following well known lemma
(see e.g.~\cite{bouchitte2008michell}):
\begin{lemma}
\label{lem:balanced}
  If $\Sigma_g(\Omega)\neq\emptyset$, then $g$ is balanced. 
\end{lemma}
\begin{proof}
Assume $\sigma\in \Sigma_g(\Omega)$.  Taking $\varphi(x)=(1,0)$ or $\varphi(x)=(0,1)$ and testing these functions against the identity $-\div
\sigma=g\H^1\ecke\partial\Omega$ (see \eqref{eq:44}), we obtain
\[
\int_{\partial\Omega} g\d\H^1=  0\,.
\]
Secondly, taking $\varphi(x)=x^\bot=(-x_2,x_1)$ as a test function, we get
\[
\int_{\partial\Omega}x^\bot \cdot g\,\d\H^1=  \int_\Omega
\left(\begin{array}{cc}0&-1\\1&0\end{array}\right) :\frac{\d\sigma}{\d|\sigma|}\d|\sigma|=0\,.
\]
The latter holds since $\sigma$ has values in the symmetric matrices.
This proves the lemma, since for every $M\in \R^{2\times 2}_{\mathrm{skew}}$,
there exists $c\in\R$ such that $Mx=c x^\bot$.
\end{proof}

For certain $h\in W^{-1,1}(\partial\Omega;\R^2)$, we now define  two integrals
$h^{(1)},h^{(2)}$. Let $x_0\in\partial\Omega$ be fixed,
$L:=\H^1(\partial\Omega)$, and let $\vartheta_{x_0}:[0,L]\to\partial\Omega$
denote the positively oriented simple Lipschitz curve that satisfies
\[
|\vartheta_{x_0}'|=1\,,\quad
\vartheta_{x_0}(0)=\vartheta_{x_0}(L)=x_0\,,\quad
\vartheta_{x_0}([0,L])=\partial\Omega\,.
\]
Obviously, $\vartheta_{x_0}|_{(0,L)}$ is a Bilipschitz homeomorphism.

For $\varphi\in L^{\infty}(\partial\Omega)$ with
$\int_{\partial\Omega}\varphi\d\H^1=0$, we may define its first integral
$\Phi(\varphi)\in  W^{1,\infty}(\partial\Omega)$ by
\[
\Phi(\varphi)(\vartheta_{x_0}(x))=\int_0^x \varphi\circ\vartheta_{x_0}(t)\d
t-c_\varphi\,,
\]
where $c_\varphi$ is chosen such that $\int_{\partial\Omega}\Phi(\varphi)\d\H^1=0$.
We may extend this definition to $h\in W^{-1,1}(\partial\Omega)=(W^{1,\infty}(\partial\Omega))'$ with
$\left<h,\chi_{\partial\Omega}\right>=0$, where $\chi_{\partial\Omega}$ is the function defined by
$\chi_{\partial\Omega}(x)=1$ for all $x\in \partial\Omega$: We let $\Phi(h)\in \mathcal
(L^\infty(\partial\Omega))'$ (to be thought of as the first integral of $h$) be defined by
\[
\left<\Phi(h),\varphi\right>=-\left<h,\Phi\left(\varphi-\fint_{\partial\Omega}\varphi\d\H^1\right)\right>\quad\text{
  for all } \varphi\in L^\infty(\partial\Omega)\,.
\]
%
For  vector valued arguments,  we may define $\Phi:
W^{-1,1}(\partial\Omega;\R^2)\to (L^\infty(\partial\Omega;\R^2))'$ by its action on
the components of its argument. 

We recall that $n$ denotes the unit outer normal of $\partial\Omega$, and
$\tau=(-n_2,n_1)$. Let these objects be understood as 
functions in $L^\infty(\partial\Omega;\R^2)$. If
$h\in  W^{-1,1}(\partial\Omega;\R^2)$ with $\left<h_i,\chi_{\partial\Omega}\right>=0$
for $i=1,2$, then $\tau\cdot \Phi(h)$ can be understood as an element of
$(L^\infty(\partial\Omega))'$, by 
\[
\left<\tau\cdot \Phi(h),\varphi\right>=\left<\Phi(h),\tau
  \varphi\right>\quad\text{ for all }\varphi\in L^\infty(\partial\Omega)\,.
\]
If we assume furthermore
$\left<\tau\cdot\Phi(h),\chi_{\partial\Omega}\right>=0$, we can define the first and second integrals
$h^{(1)}\in  (L^\infty(\partial\Omega;\R^2))'$ and $h^{(2)}\in
\Phi((L^\infty(\partial\Omega))')$  by
\begin{equation}
  \label{eq:35}
  \begin{split}
    h^{(1)}&=\Phi(h)\\
    h^{(2)}&=\Phi(\tau\cdot\Phi(h))\,.
  \end{split}
\end{equation}

In order to make the transition between stresses and their Airy potentials, the
following definition will be convenient: Let 
\[
X:=\left\{g\in W^{-1,1}(\partial\Omega;\R^2): n\cdot (g^\bot)^{(1)}\in L^1(\partial\Omega),
\,(g^\bot)^{(2)}\in \gamma_0(W^{2,1}(\Omega))\right\}\,.
\]
We make $X$ into a topological vector space by letting the topology on $X$ be
the strongest one 
that makes the following map continuous:
\[
\begin{split}
  X&\to L^1(\partial\Omega)\times\gamma_0(W^{2,1}(\Omega))\\
  g&\mapsto (n\cdot (g^\bot)^{(1)},(g^\bot)^{(2)})\,.
\end{split}
\]
\begin{remark}
\label{rem:balanced}
  \begin{itemize}
\item[(i)] The requirements for the existence of $(g^\bot)^{(1)},(g^\bot)^{(2)}$, namely that
  \begin{equation}
\label{eq:20}
\left<g_i,\chi_{\partial\Omega}\right>=0\quad\text{ for }i=1,2\,,\qquad
\left<\tau\cdot\Phi(g^\bot ),\chi_{\partial\Omega}\right>=0\,,
\end{equation}
precisely express that $g$ has to be balanced, $\int_{\partial\Omega}(Mx+b)\cdot
g\d\H^1(x)=0$ for all $M\in \R^{2\times 2}_{\mathrm{skew}}$, $b\in\R^2$.
To see that $\int_{\partial\Omega}{Mx}\cdot g\d \H^1=0$ for all $M\in
  \R^{2\times 2}_{\mathrm{skew}}$ is equivalent with
  the second equation in \eqref{eq:20}, we observe that
\[
\begin{split}
  \int_{\partial\Omega}\left(\begin{array}{cc}0&-1\\1&0\end{array}\right)x\cdot
  g\d \H^1&=-\int_{\partial\Omega}x\cdot
  g^\bot\d \H^1\\
  &=-\int_{\partial\Omega}\left(x_0+\int_{0}^{\gamma_{x_0}^{-1}(x)}\tau\d\H^1\right)\cdot
  g^\bot(x)\d\H^1(x)\\
  &=-\left<\Phi(g^\bot),\tau\right>=-\left<\tau\cdot\Phi(g^\bot),\chi_{\partial\Omega}\right>\,.
\end{split}
\]
\item[(ii)] The space $X$ is a replacement for $W^{-1,1}(\partial\Omega;\R^2)$; the
    latter is slightly too large for our purposes. In the formulation of the
    compliance minimization problem via Airy potentials, we need to translate
    the integrals $(g^\bot)^{(2)},(g^\bot)^{(1)}\cdot n$ of the boundary values $g\in X$ into
    boundary values of a function in $BH(\Omega)$ and its normal
    derivative. This is not possible for $W^{-1,1}(\partial\Omega;\R^2)$,
    firstly because $L^1(\partial\Omega)\subsetneq
    \Phi(W^{-1,1}(\partial\Omega))$, and secondly because
    $\gamma_0(BH(\Omega))=\gamma_0(W^{2,1}(\Omega))\subsetneq
    W^{1,1}(\partial\Omega)=\Phi(L^1(\partial\Omega))$. Nevertheless, with our
    choice of $X$ we have $H^{-1/2}(\partial\Omega;\R^2)\subset X$ and
    furthermore $X$ contains  balanced finite sums of delta distributions that
    do not contain applied forces tangential to $\partial \Omega$. For the
    precise statement,  see Lemma \ref{lem:deltabdry} below.
  \end{itemize}
\end{remark}

The upcoming lemma only serves to prove the claim made in the previous remark
and can be skipped by the reader who is only interested in the statement and proof of the main
theorem. 

\medskip

For $v\in \R^2$
and $x\in \partial\Omega$, we write $v\in \mathcal T_x(\partial\Omega)$ if there exists $\e>0$ such
that either $\{x+tv:t\in(-\e,\e)\}\cap\Omega=\emptyset$ or $\{x+tv:t\in(-\e,\e)\}\subset\overline{\Omega}$. This is the case, for example,
if $v$ is the tangent vector to $\partial \Omega$ in a point $x$ where the
curvature does not change sign. (If the curvature changes sign at $x$, then $\mathcal T_{x}(\partial\Omega)=\emptyset$.) If $\partial\Omega$ is not $C^2$ near $x$, then
the set $\mathcal T_x(\partial\Omega)$ is larger, see Figure \ref{fig:tcal}.

\begin{figure}[h]
\includegraphics{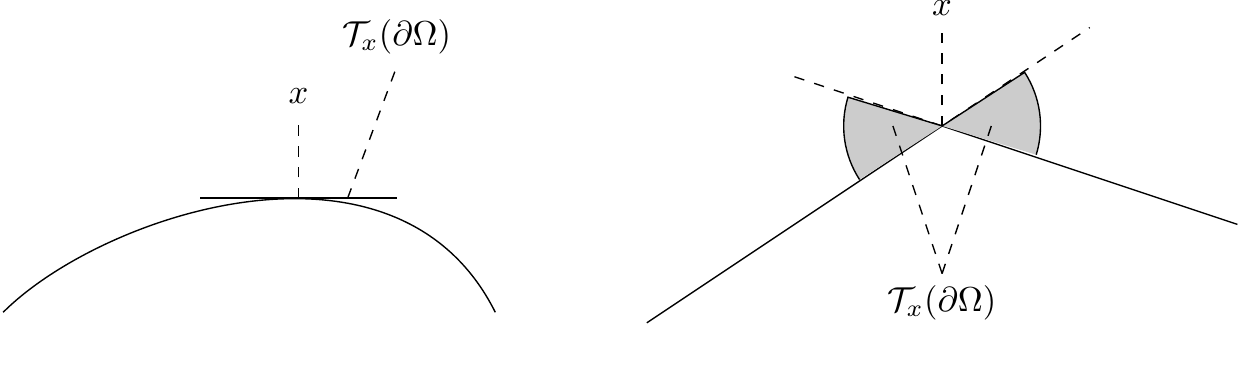}
\caption{In the left panel, the space $\mathcal T_{x}(\partial\Omega)$ is just
  the tangent space at $x\in\partial\Omega$. In the right panel, the space $\mathcal
  T_{x}(\partial\Omega)$ is a cone. \label{fig:tcal}}
\end{figure}

\begin{lemma}
\label{lem:deltabdry}
For $i=1,\dots,N$ let $x_i\in\partial\Omega$ and $v_i\in\R^2$ such that
$v_i\not\in \mathcal T_{x_i}(\partial\Omega)$, and additionally
\[
\sum_{i}v_i=\sum_{i}v_i\cdot x_i^\bot=0\,.
\]
Then $g=\sum_{i}\delta_{x_i} v_i$ is balanced, and $g\in X$.
\end{lemma}
\begin{proof}
Using the definitions, the fact that $g$  is balanced is obvious.
With the notation we have introduced above, and the assumption $x_0\neq x_i$
for $i=1,\dots,N$, we have
\[
(g^\bot)^{(1)}(\vartheta_{x_0}(x))=\sum_{i:\vartheta_{x_0}^{-1}(x_i)< x} v_i^\bot\,.
\]
Hence $(g^\bot)^{(1)}$ is a piecewise constant function
$\partial\Omega\to\R^2$. Let $F_i\in\R^2$ denote the value of $(g^\bot)^{(1)}$
on the arc connecting $x_i$ and $x_{i+1}$ in $\partial \Omega$
(counterclockwise). Furthermore we note that $(g^\bot)^{(2)}\in W^{1,\infty}(\partial\Omega)$
with $\partial_\tau  (g^\bot)^{(2)}=(g^\bot)^{(1)}\cdot \tau$. 
For every $i=1,\dots,N$, choose $\e_i>0$ at least so small that
\[
\begin{split}
  [x_i-\e_iv_i,x_i+\e_iv_i]\cap \partial\Omega&= \{x_i\}\\
  [x_i-\e_iv_i,x_i+\e_iv_i]\cap[x_j-\e_jv_j,x_j+\e_jv_j]&=\emptyset\quad\text{
    for }i\neq j\,.
\end{split}
\]
For $i=1,\dots,N$, there is exactly  one out of the two points $x_i\pm\e_iv_i$
that is contained in $\Omega$. Denote this point by $\bar x_i$. 
Let $\tilde\Omega\Subset\Omega$ be a simply connected polygonal domain with
$\bar x_i\in \partial\tilde\Omega$ for $i=1,\dots,N$.
Let $Q_i$ denote the open subset of $\Omega\setminus\tilde\Omega$ that is bounded by
$\partial\Omega$, $[x_i,\bar x_i]$, $\partial\tilde\Omega$ and $[x_{i+1},\bar
x_{i+1}]$, see Figure \ref{fig:omtil}.

\begin{figure}[h]
\includegraphics{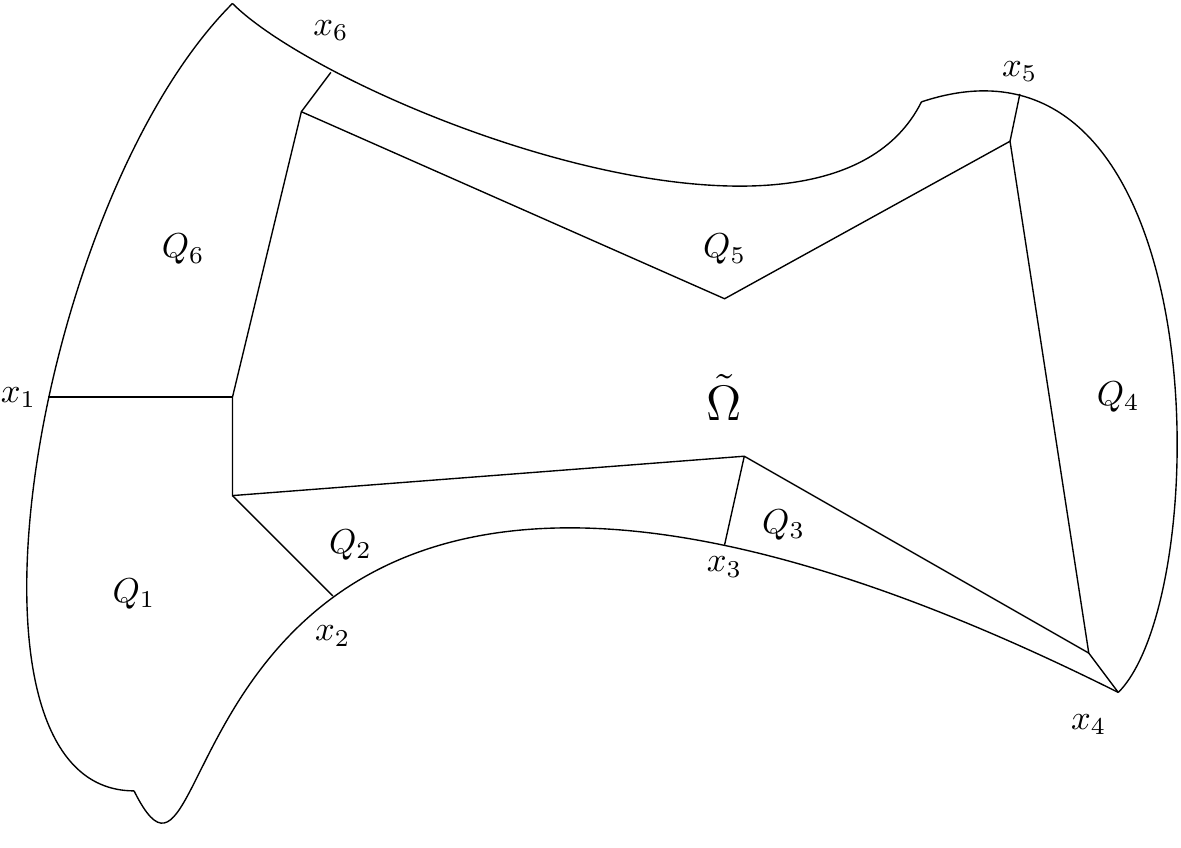}
\caption{The partition of $\Omega$ into subsets $Q_i$ and a polygonal set
  $\tilde \Omega$.\label{fig:omtil}}
\end{figure}

On  $\Omega\setminus\tilde\Omega$, we may define  $u\in BH(Q)$ almost everywhere
by setting
\[
u(x):= (g^\bot)^{(2)}(x_i)+ F_i\cdot (x-x_i)\quad \text{ if }x\in Q_i\,.
\] 
This makes $u$ affine on each subset $Q_i$, and we claim that there exists a
continuous extension to $\Omega\setminus \tilde\Omega$. Indeed, we need to check continuity only at the boundaries $\partial
Q_i\cap\partial Q_{i+1}=[x_{i+1},\bar x_{i+1}]=[x_{i+1},x_{i+1}\pm
\e_{i+1}v_{i+1}]$. For $x\in [x_{i+1},\bar x_{i+1}]$, we have
\[
\begin{split}
  \lim_{\substack {z\to x\\z\in Q_i}}u(z)-
  \lim_{\substack {z\to x\\z\in Q_{i+1}}}u(z)
  &=\left( u(x_{i+1})+F_i\cdot (x-x_{i+1})\right)-\left( u(x_{i+1})+F_{i+1}\cdot
    (x-x_{i+1})\right)\\
  &= (F_i-F_{i+1})\cdot (x-x_{i+1})\\
  &=-v_{i+1}^\bot\cdot \tilde \e \,v_{i+1}\quad\text{ for some
  }\tilde\e\in[-\e_{i+1},\e_{i+1}]\\
  &=0\,.
\end{split}
\]
This proves the existence of the continuous extension of $u$ to
$\Omega\setminus\tilde\Omega$. Let $T$ be a triangulation of $\tilde\Omega$, and extend $u$ to a function
that is affine on each triangle of $T$ and continuous on $\Omega$.  Since piecewise
affine continuous functions are of bounded Hessian,
we have $u\in BH(\Omega)$, and 
\[
\begin{split}
  \gamma_0(u)&=(g^\bot)^{(2)}\in \gamma_0(BH(\Omega))=\gamma_0(W^{2,1}(\Omega))\\
  \gamma_1(u)&= (g^\bot)^{(1)}\cdot n\in L^1(\partial\Omega)
\end{split}
\]
Hence $g\in X$, and the lemma is proved.
   \end{proof}

For  $\sigma\in C^0(\overline\Omega;\Rsym)$, we have that $\sigma\in \Sigma_g(\Omega)$ implies
$\sigma\cdot n=g$ on $\partial\Omega$. If additionally
$\sigma=\cof \nabla^2 u$ for some $u\in C^2(\overline{\Omega})$, then $g^\bot=(\cof \nabla^2u\cdot n)^\bot=-\nabla^2
u\cdot \tau=-\partial_\tau\nabla u$. This implies that the integral
$(g^\bot)^{(1)}$ is equal to $-\nabla u$ up to a  constant, and $(g^\bot)^{(2)}$ is
equal to $-u$ up to an affine function. The following lemma restates these observations
for the non-smooth case.

\begin{lemma}
\label{lem:gaffine}
Let  $g\in W^{-1,1}(\partial\Omega;\R^2)$, $\sigma \in \Sigma_g(\Omega)$, and let
  $U$ be a neighborhood of $\overline\Omega$, such that $u\in W^{1,1}(U)$,  and
  $D^2u\ecke \overline\Omega=\cof\sigma$. 
Then there exists $\zeta\in L^1(\partial\Omega)$ and an affine function $F$ such that
\[
\begin{split}
  (g^\bot)^{(1)}\cdot n &=-\gamma_1(u)-\zeta+\nabla F\cdot n\\
  (g^\bot)^{(2)}&=-\gamma_0(u)+F\,.
\end{split}
\]
The same conclusion holds true if $g\in H^{-1/2}(\partial\Omega;\R^2)$,
$\sigma\in L^2(\Omega;\Rsym)$, and $u\in H^2(\Omega)$ with $\sigma=\cof\nabla^2 u$.
\end{lemma}

\begin{proof}
To prove the first claim, we show that there exists $\zeta\in
L^1(\partial\Omega)$ and some vector $c\in\R^2$ such that
\begin{equation}
(g^\bot)^{(1)}=-\nabla u|_{\partial\Omega}-\zeta n+c\,,\label{eq:5}
\end{equation}
where the right hand side is understood in the sense of traces of $BV$
functions. To prove this claim, let $\varphi\in C^1(U)$. Then we have
\[
\begin{split}
\left<\partial_\tau\varphi,(g^\bot)^{(1)}\right>&=-  \left<\varphi,g^\bot\right>\\
&=- \int_{\overline\Omega} \left(\nabla\varphi\cdot\d\sigma\right)^\bot \\
  &=-\int_{\overline\Omega} \nabla\varphi\cdot\d(\cof D^2 u) \\
    &=\int_{\overline\Omega} (\nabla\varphi)^\bot\cdot\d(D^2 u)\,.
    \end{split}
\]
Let $\mu=\frac{\d(D^2 u)}{\d \H^1} \H^1\ecke\partial \Omega$ denote the
restriction of the jump part
of $D^2u$ to $\partial\Omega$.
Then we have
\[
D^2u\ecke\overline{\Omega}=D^2 u\ecke \Omega+\mu\,.
\]
Using  Theorem \ref{thm:alberti} and the symmetry of $D^2 u$, we have that $\mu$
is $\H^1$-almost everywhere parallel to $n\otimes n$, and we may write
$\mu=\zeta n\otimes n\H^1$. Hence,
\[
\int_{\overline\Omega} (\nabla\varphi)^\bot\cdot\d\mu
=-\int_{\partial\Omega }\zeta n \partial_\tau\varphi\d\H^1\,.
\]
By the Gauss' Theorem for $BV$ functions (see e.g.~\cite{MR0638362}), we have that
\[
\begin{split}
  \int_{\overline\Omega} (\nabla\varphi)^\bot\cdot\d(D^2 u\ecke \Omega)
  &=\int_{\partial\Omega}(\nabla u\otimes \nabla\varphi^\bot)\cdot n\d\H^1\\
  &=\int_{\partial\Omega}-\nabla u \partial_\tau \varphi\d\H^1\,.
\end{split}
\]
Hence, we have
\[
-\int_{\partial\Omega}(\nabla u+ \zeta n)\partial_\tau \varphi\d\H^1
=\int_{\partial\Omega}(g^\bot)^{(1)} \partial_\tau \varphi\d\H^1\quad\text{ for
  all }\varphi\in C^1(U)\,.
\]
This proves \eqref{eq:5} and hence the first claim. Next we
observe that
\[
\begin{split}
  \partial_\tau \left((g^\bot)^{(2)}+\gamma_0(u)\right)&= 
  \left((g^\bot)^{(1)}+\nabla u|_{\partial\Gamma}+\zeta n\right)\cdot \tau\\
  &=c\cdot \tau\,,
\end{split}
\]
which proves the second claim. Finally, the situation $u\in H^2(\Omega)$ is just
a special case of what we have just proved, by extending $u$ to some
 $\tilde u\in H^2(U)$, where $U$ is some neighborhood of $\overline\Omega$, and  $\tilde u|_{\Omega}=u$, which is possible
by Theorem \ref{thm:steinext}.
\end{proof}





\subsection{Statement of the main theorem}

First we will state a  proposition that is basically equivalent to our main theorem, using Airy potentials and the most general
boundary values that are allowed within this framework.

For $\lambda>0, \bar f=(\bar f_1,\bar f_2)\in H^{3/2}(\partial\Omega)\times H^{1/2}(\partial\Omega)$ let the functional
$\F_{\bar f,\lambda}:H^{2}(\Omega)\to \R$ be defined by
\[
\F_{\bar f,\lambda}(u)=\begin{cases}\lambda^{-1/2}\int_{\Omega}F_\lambda(\nabla^2 u)\d
  x&
  \text{ if }  \gamma_0(u)=\bar f_1\text{ and }\gamma_1( u) =\bar f_2\\
+\infty & \text{ else.}\end{cases}
\]
Furthermore, for $\bar f=(\bar f_1,\bar f_2)\in \gamma_0(W^{2,1}(\Omega))\times L^1(\partial\Omega)$, let the functional
$\F_{\bar f}:BH(\Omega)\to \R$ be defined by
\[
\F_{\bar f}(u)=\begin{cases}2\rho^0(D^2
  u)(\Omega)+2\int_{\partial\Omega}|\gamma_1(u)-\bar f_2|\d \H^1&
  \text{ if }  \gamma_0(u)=\bar f_1\\
+\infty & \text{ else.}\end{cases}
\]

In the statement of the following theorem, we use the standard norm on the
Cartesian product 
$H^{3/2}(\partial\Omega)\times H^{1/2}(\partial\Omega)$, i.e.
\[
\|(\bar f_1,\bar f_2)\|_{H^{3/2}(\partial\Omega)\times H^{1/2}(\partial\Omega)}:=
\|\bar f_1\|_{H^{3/2}(\partial\Omega)}+\|\bar f_2\|_{H^{1/2}(\partial\Omega)}\,.
\]

\begin{proposition}
 Let $\Omega\subset\R^2$  satisfy Definition \ref{def:H1}, and assume that
\newcommand{\fl}{\bar f_\lambda}
\newcommand{\hl}{\bar h_\lambda}
\begin{equation}
\begin{split}
  \fl&\in H^{3/2}(\partial\Omega)\times H^{1/2}(\partial\Omega)\\
 \bar f&\in  \gamma_0(W^{2,1}(\Omega))\times L^1(\partial\Omega)\\
\fl&\to \bar f \text{ weakly in }\gamma_0(W^{2,1}(\Omega))\times L^1(\partial\Omega)\\
 \frac{\|\fl\|_{H^{3/2}(\partial\Omega)\times H^{1/2}(\partial\Omega)}
}{\lambda^{1/4}}&\to
  0
\text{ as }\lambda\to \infty\,.
\end{split}\label{eq:25}
\end{equation}
\label{prop:gammaairy}
  \begin{itemize}
  \item[(i)] \emph{Compactness:} Let $\{u_\lambda\}_\lambda\subset H^{2}(\Omega)$ be
  such that $\F_{\fl,\lambda}(u_\lambda)<C$. Then there exists a subsequence (no
  relabeling) and $u\in BH(\Omega)$ such that $u_\lambda\to  u$ weakly
  * in $BH(\Omega)$.
\item[(ii)] \emph{Lower bound:} If $u_\lambda\to u$ weakly * in
  $BH(\Omega)$,  then
  \begin{equation}
  \liminf_{\lambda\to\infty}\F_{\fl,\lambda}( u_\lambda)\geq \F_{\bar f}( u)\,.\label{eq:26}
  \end{equation}

\item[(iii)] \emph{Upper bound:} For
  every $u\in BH(\Omega)$ there exists a sequence
  $\{ u_\lambda\}_\lambda\subset H^{2}(\Omega)$ such that $ u_\lambda\to  u$
  weakly * in $BH(\Omega)$
   and $\lim_{\lambda\to\infty}\F_{\fl,\lambda}(u_\lambda)=\F_{\bar f}( u)$.
  \end{itemize}
\end{proposition}
  
With the proposition at hand, we can prove our main theorem, which contains Theorem
\ref{thm:gammabdry} as a special case.

\begin{theorem}
Let $\Omega\subset\R^2$  satisfy Definition \ref{def:H1}, and  assume that
\label{thm:main}
 \begin{equation}
\begin{split}
  g_\lambda&\in H^{-1/2}(\partial\Omega;\R^2)\\
g\in W^{-1,1}(\partial\Omega;\R^2),\quad(g^\bot)^{(2)}&\in \gamma_0(W^{2,1}(\Omega))\\
g_\lambda^\bot&\to g^\bot \text{ weakly in }X\\
   \lambda^{-1/2}\|\ul\|^2_{H^{-1/2}(\partial\Omega;\R^2)}&\to
  0\text{ as }\lambda\to \infty\,.
\end{split}\label{eq:42}
\end{equation}
  \begin{itemize}
  \item[(i)] \emph{Compactness:} Let $\{\sigma_\lambda\}_\lambda\subset L^2(\Omega;\Rsym)$ be
  such that $\G_{g_\lambda,\lambda}(\sigma_\lambda)<C$. Then there exists a subsequence (no
  relabeling) and $\sigma\in \M(\Omega;\Rsym)$ such that $\sigma_\lambda\to \sigma$ weakly * in
  the sense of measures. 
\item[(ii)] \emph{Lower bound:} If $\sigma_\lambda\to \sigma$ weakly * in the sense of
  measures, then
  \begin{equation}
  \liminf_{\lambda\to\infty}\G_{g_\lambda,\lambda}(\sigma_\lambda)\geq \G_{g}(\sigma)\,.\label{eq:40}
  \end{equation}

\item[(iii)] \emph{Upper bound:} Assume that $\Omega$ contracts nicely. For every $\sigma\in\M(\Omega;\Rsym)$ there exists a sequence
  $\{\sigma_\lambda\}_\lambda\subset L^2(\Omega;\Rsym)$ such that $\sigma_\lambda\to \sigma$
  weakly * in the sense of measures and $\lim_{\lambda\to\infty}\G_{g_\lambda,\lambda}(\sigma_\lambda)=\G_{g}(\sigma)$.
  \end{itemize}

\end{theorem}

\begin{remark}
  The reason for the assumption
  $\lambda^{-1/2}\|g_\lambda\|_{H^{-1/2}(\partial\Omega;\R^2)}^2\to 0$ here, and
  for the analogous assumption
  $\lambda^{-1/4}\|\bar f_\lambda \|_{H^{3/2}(\partial\Omega)\times
    H^{1/2}(\partial\Omega)}\to 0$ in the Proposition, is  technical. This allows us
  to control the behavior of boundary layers in the upper bound construction,
  and gives a convenient estimate for error terms appearing in the proof of the
  lower bound, see the proof of Proposition \ref{prop:gammaairy}. However, these
  assumptions are not a restriction, in the sense that every $g\in W^{-1,1}(\partial\Omega;\R^2)$ with $(g^\bot)^{(2)}\in
  \gamma_0(W^{2,1}(\Omega))$ possesses an approximating sequence $g_\lambda$ with
  these properties.
\end{remark}

\begin{proof}
Recall the definition of the
``integrals'' $(g^\bot)^{(1)},(g^\bot)^{(2)}$ from \eqref{eq:35}.
We may assume that $g_\lambda, g$ are balanced, since otherwise
$S_{g_\lambda}=\emptyset$ or $\Sigma_g=\emptyset$ by Lemma
\ref{lem:balanced}. By Remark \ref{rem:balanced} (i), the requirements for the
existence of $(g^\bot)^{(1)},(g^\bot)^{(2)}$, $(g_\lambda^\bot)^{(1)},(g_\lambda^\bot)^{(2)}$ are met.
Now we set
\begin{equation}
  \begin{split}
    \bar f_\lambda:=&-((g_\lambda^\bot)^{(2)},(g_\lambda^\bot)^{(1)}\cdot n)\\
    \bar f:=&-((g_\lambda^\bot)^{(2)},(g_\lambda^\bot)^{(1)}\cdot n)\,.
  \end{split}\label{eq:31}
\end{equation}
With these definitions, \eqref{eq:42} implies \eqref{eq:25}. 

For the compactness part, assume that
$\G_{g_\lambda,\lambda}(\sigma_\lambda)<C$. Using Lemma \ref{lem:airylem}, we
obtain a sequence $u_\lambda$ in $H^2(\Omega)$ with
$\G_{g_\lambda,\lambda}(\sigma_\lambda)=\F_{\bar
  f_\lambda,\lambda}(u_\lambda)$. By Proposition \ref{prop:gammaairy}, we obtain weak
* convergence of $\sigma_\lambda=\cof D^2 u_\lambda$ to $\cof D^2 u$.

For the lower bound part, we assume $\{\sigma_\lambda\}_\lambda\subset
L^2(\Omega;\Rsym)$ with $\sigma_\lambda\to \sigma$ weakly * in
$\M(\Omega;\Rsym)$ and $\lim_{\lambda\to \infty}\G_{g_\lambda,\lambda}(\sigma_\lambda)=M<\infty$. We may assume $-\div\sigma_\lambda=g_\lambda
\H^1\ecke\partial\Omega$, since otherwise
$\G_{g_\lambda,\lambda}(\sigma_\lambda)=+\infty$. Using Lemma \ref{lem:airylem},
let $u_\lambda\in H^{2}(\Omega)$ and $u\in BH(\Omega)$ such that $\cof
\nabla^2u_\lambda=\sigma_\lambda$ and $\cof D^2 u=\sigma$. By Lemma
\ref{lem:gaffine}, we may assume that by the addition of 
suitable affine functions, we have 
\[
\begin{split}
  \gamma_0(u_\lambda)= -(g_\lambda^\bot)^{(2)},\quad
  \gamma_1(u_\lambda)= -(g_\lambda^\bot)^{(1)}\cdot n\,,\\
  \gamma_0(u)=-( g^\bot)^{(2)},\quad \gamma_1(u)=-( g^\bot)^{(1)}\cdot n\,.
\end{split}
\]
This implies $\G_{g_\lambda,\lambda}(\sigma_\lambda)=\F_{\bar
  f_\lambda,\lambda}(u_\lambda)$ and $\G_{g}(\sigma)=\F_{\bar
  f}(u_\lambda)$, and hence the lower bound follows from the lower bound part of
Proposition \ref{prop:gammaairy}.

For the upper bound part, let $\sigma\in \Sigma_g(\Omega)$. Applying
Lemma \ref{lem:airylem} and Lemma \ref{lem:gaffine} we have that after the
addition of an affine function, $u\in
BH(\Omega)$ with $(\gamma_0(u),\gamma_1(u))=\bar f$. The  upper bound part of
Proposition \ref{prop:gammaairy} and \eqref{eq:31}
immediately yield the desired recovery sequence.
\end{proof}

The rest of this paper is concerned with the proof of Proposition
\ref{prop:gammaairy}. 




\section{Proof of compactness and upper bound}
\label{sec:proof-comp-upper}
In this section and the next, we use the notation
\[
G_\lambda(\xi)=\lambda^{-1/2}Q_2F_\lambda(\xi)\,.
\]
By Theorem \ref{thm:Q2F}, we have
\[
G_\lambda(\xi)=\begin{cases}2\left(\rho^0(\xi)-\lambda^{-1/2}|\det\xi|\right)&\text{ if
  }\rho^0(\xi)\leq \sqrt{\lambda}\\ \lambda^{1/2}+\lambda^{-1/2}|\xi|^2& \text{else. }\end{cases}
\]
\begin{proof}[Proof of compactness in Proposition \ref{prop:gammaairy}]
We claim that for $\xi\in \Rsym$, 
\begin{equation}
  \label{eq:3}
  \frac12\rho^0(\xi)\leq G_\lambda(\xi)\,.
\end{equation}
By $|\xi|\leq \rho^0(\xi)$, this implies 
\begin{equation}
\frac12|D^2 u_\lambda|(\Omega)=\frac12\|\nabla^2 u_\lambda\|_{L^1(\Omega)}\leq \G_\lambda(u_\lambda)\,.
\end{equation}
By Lemma \ref{lem:BHpoincare}, we obtain that  $u_\lambda$ is a bounded sequence
in $BH$. By  Theorem \ref{thm:BHcompact} we obtain that  a subsequence converges
weakly * in $BH(\Omega)$.  It remains to prove \eqref{eq:3}.

Let $x\in \Omega$. Let $a_1,a_2$ denote the absolute values of the eigenvalues
of $\xi$. For $\rho^0(\xi)=a_1+a_2 \leq \sqrt{\lambda}$ we have
\[
\begin{split}
  a_1+a_2-2\lambda^{-1/2}a_1a_2
  \geq \lambda^{-1/2}\left( (a_1+a_2)^2-2a_1a_2\right)
\geq 0\,.
\end{split}
\]
Hence we have 
\begin{equation}
G_\lambda(\xi)=2(a_1+a_2-\lambda^{-1/2}a_1a_2)\geq a_1+a_2=\rho^0(\xi)\,.\label{eq:2}
\end{equation}
If
$a_1+a_2\geq \sqrt{\lambda}$, then 
\[
G_\lambda(\xi)\geq \frac{a_1^2+a_2^2}{\sqrt{\lambda}}\geq
\frac12\frac{(a_1+a_2)^2}{\sqrt{\lambda}}\geq \frac12 (a_1+a_2)=\frac12 \rho^0(\xi)
\]
Combining these two cases, we obtain \eqref{eq:3} which  proves the
compactness part of the theorem.
\end{proof}

\begin{proof}[Proof of the upper bound in Proposition \ref{prop:gammaairy}]
By Theorem \ref{thm:traceop}, the application of the right inverse of the trace
operator $(\gamma_0,\gamma_1)$ to $\bar f_\lambda$ and $\bar
f$ yields a sequence
$ f_\lambda:=(\gamma_0,\gamma_1)^{-1}\bar f_\lambda$ in
$H^{2}(\Omega)$ and $f:=(\gamma_0,\gamma_1)^{-1}\bar f\in W^{2,1}(\Omega)$ such
that
 \begin{equation}
\begin{split}
   f_\lambda&\to f\quad\text{ weakly in }W^{2,1}(\Omega)\\
  \lambda^{-1/2}\| f_\lambda\|^2_{H^{2}(\Omega)}&\to 0\,.
\end{split}\label{eq:21}
\end{equation}
By Theorem \ref{thm:steinext}, we may extend $ f_\lambda$ and $ f$
to all of $\R^2$ such that
\[ 
\begin{split}
   f_\lambda&\to f\quad\text{ weakly in }W^{2,1}(\R^2)\\
  \lambda^{-1/2}\| f_\lambda\|^2_{H^{2}(\R^2)}&\to 0\,.
\end{split}
\]

Let $v\in BH(\R^2)$ be defined by 
\[
v(x)=\begin{cases}u(x)-f(x) &\text{ for }x\in \Omega\\ 0&\text{
    else.}\end{cases}
\]
Let $V\subset\R^2$ be some neighborhood of $\overline\Omega$, and let
$\phi:V\to \phi(V)\subset\R^2$ be a $C^2$-diffeomorphism, such that
$\phi(\Omega)=K$, where $K$ is a convex polygon that contains the origin. Such a map $\phi$  exists by
our assumptions on $\Omega$, see Definition \ref{def:H1}. Choose $C_1>0$ such
that
\[
\dist
\left(\phi^{-1}\left(\frac{K}{\sqrt{1+C_1\e}}\right),\partial\Omega\right)>\e
\]
for all $\e>0$ small enough.
For such $\e$,   we set
\[
\Theta_\e(x)=\phi^{-1}\left(\frac{\phi(x)}{\sqrt{1+C_1\e}}\right)\,,
\]
and
\[
v_\e(x):=v\left(\Theta_\e(x)\right)\,.
\]
Note that $v_\e\in BH(\R^2)$
and $v_\e=0$ on $\{x:\dist(x,\partial\Omega)\leq \e\}$.
Additionally, we claim that in the limit $\e\to 0$, we have
\begin{equation}
\begin{split}
  v_\e&\to v  \quad\text{ in } W^{1,1}(\R^2)\,,\\
  |D^2 v_\e|(\R^2)&\to |D^2 v|(\R^2)\,.
\end{split}\label{eq:50}
\end{equation}
To prove this claim, we 
observe that $\nabla\Theta_\e(x)\to\id_{2\times 2}$ and $\nabla^2\Theta_\e(x)\to 0$
uniformly as $\e\to 0$.
Now we have
\[
\begin{split}
  \nabla v_\e(x)&=\nabla v(\Theta_\e(x))\nabla\Theta_\e(x)\\
D^2 v_\e&= (D^2 v\circ\Theta_\e):\left(\nabla\Theta_\e\otimes\nabla\Theta_\e\right)\\
&\quad+\nabla v\circ\Theta_\e\cdot \nabla^2\Theta_\e\, \L^2 \,,
\end{split}
\]
where by $(D^2 v\circ\Theta_\e)$, we mean the Radon measure that is defined by 
\[
\int_\Omega \varphi:\d(D^2 v\circ\Theta_\e)=\int_{\Omega}
\varphi(\Theta_\e^{-1}(z))\det \nabla \Theta_\e^{-1}(z):\d(D^2 v)(z)\quad\text{
  for }\varphi\in C_c^0(\Omega;\R^{2\times 2}).
\]
From the uniform convergences $(\Theta_\e(x)-x)\to 0$, $\nabla\Theta_\e(x)\to
\id_{2\times 2}$ and $\nabla^2\Theta_\e(x)\to 0$, the claim \eqref{eq:50}
follows.

\medskip

We let $\varphi\in C^\infty_c(\R^2)$ be
such that $\int\varphi(x)\d x=1$ and $\supp \varphi\subset \{x\in\R^2:|x|<1\}$, and $\varphi_\e=\e^{-2}\varphi(\cdot/\e)$.
Next we set
\[
\tilde v_\e:=\varphi_\e * v_\e\,,
\]
and by the properties of $\varphi_\e,v_\e$, we have
\[
\begin{split}
  \tilde v_\e&=0, \quad\nabla \tilde v_\e=0\quad\text{ on } \partial \Omega\,,\\
\tilde v_\e&\to v  \quad\text{ in } W^{1,1}(\R^2)\,,\\
\|\nabla^2 v_\e\|_{L^1(\R^2)}&\to |D^2 v|(\R^2)\,.
\end{split}
\]
Furthermore, there exists a constant $C_2>0$ such that
\[
  \|\nabla^2 v_\e\|_{L^\infty(\R^2)}\leq C_2 \e^{-1}|D^2 v|(\R^2)\,.
\]
We set
\[
\e(\lambda):= 4\lambda^{-1/2}C_2 |D^2 v|(\R^2)\,,
\]
and define the recovery sequence by
\[
u_\lambda= f_{\lambda}+ \tilde v_{\e(\lambda)}\,.
\]
Our choice of $\e(\lambda)$ implies 
 that
\begin{equation}
\|\nabla^2 \tilde v_{\e(\lambda)}\|_{L^\infty}<\frac{\sqrt{\lambda}}{4}\,.\label{eq:37}
\end{equation}
From \eqref{eq:21} and Theorem \ref{thm:BHcompact}, we see that $u_\lambda$
converges weakly * in $BH(\R^2)$ to the function
\[
\tilde u(x)=\begin{cases} u(x)&\text{ if }x\in\Omega\\
f(x) &\text{ else.}\end{cases}
\]
Next let 
\[
\tilde E_\lambda:=\{x\in\Omega:\rho^0(\nabla^2u_\lambda)>\sqrt{\lambda}\}\,.
\]
By \eqref{eq:30} and \eqref{eq:37}, we have
\[
\tilde E_\lambda\subset\{x\in\Omega:|\nabla^2f_\lambda|>\sqrt{\lambda}/4\}=:E_\lambda\,.
\]
Let 
\[
\Omega_\e:=\{x\in\R^2:\dist(x,\Omega)<\e\}\,.
\]
For every $\e>0$, we have that 
\[
\begin{split}
  \int_{\Omega_\e\setminus E_\lambda} G_\lambda(\nabla^2 u_\lambda)\d
  x& \leq 2\int_{\Omega_\e\setminus E_\lambda} \rho^0(\nabla^2 u_\lambda)\d x\\
  \int_{ E_\lambda} G_\lambda(\nabla^2 u_\lambda)\d x& \lesssim
  \lambda^{1/2}\L^2(E_\lambda)+\int_{E_\lambda} \frac{|\nabla^2
      u_\lambda|^2}{\sqrt{\lambda}}\d x\\
&\,\,\lesssim \frac{\|\nabla^2 f_\lambda\|^2_{L^2}}{\sqrt{\lambda}}\to
0\quad\text{ as } \lambda\to \infty\,,
  \end{split}
\]
where we have used the assumption \eqref{eq:25}.
This implies, using the non-negativity of $G_\lambda$,
\begin{equation}
\begin{split}
  \limsup_{\lambda\to\infty}\int_{\Omega} G_\lambda(\nabla^2
  u_\lambda)\d x
&\leq
\limsup_{\lambda\to\infty}\int_{\Omega_\e} G_\lambda(\nabla^2
  u_\lambda)\d
  x\\
&\leq \limsup_{\lambda\to\infty}2\int_{\Omega_\e} \rho^0(\nabla^2 u_\lambda)\d x\\
  &=2\int_{\Omega_\e}\rho^0(D^2 \tilde u)\,,
\end{split}
\label{eq:39}
\end{equation}
where we have used Theorem \ref{thm:homomeas} in the last step, and the fact
that
\[
\rho^0(D^2 \tilde u)(\partial \Omega_\e)=0\,.
\]
Taking the limit $\e\to 0$ in the estimate \eqref{eq:39}, we obtain
\begin{equation}
\begin{split}
  \limsup_{\lambda\to\infty}\int_{\Omega} G_\lambda(\nabla^2
  u_\lambda)\d x&\leq
  2\int_{\overline \Omega}\rho^0(D^2 \tilde  u)\\
  &=2\int_{\Omega}\rho^0(D^2 u)+2\int_{\partial\Omega}|\nabla u-\nabla
  f|\d\H^1\,.
\end{split}\label{eq:36}
\end{equation}
On the right hand side above,  $\nabla u|_{\partial\Omega}$ has to be understood
as the trace of $\nabla u\in BV(\Omega)$. 
By Theorem \ref{thm:BHcont}, we have that $\tilde u$ is continuous. In
particular, we must have $f|_{\partial \Omega}=\gamma_0(u)$, and hence

\begin{equation}
\begin{split}
  \int_{\partial\Omega}|\nabla u-\nabla
  f|\d\H^1&=\int_{\partial\Omega}|\gamma_1(u)-\partial_n f|\d\H^1\\
  &=\int_{\partial\Omega}|\gamma_1(u)-\bar f_2|\d\H^1\,,
\end{split}\label{eq:46}
\end{equation}
which implies
\[
\limsup_{\lambda\to\infty}\int_{\Omega} G_\lambda(\nabla^2
  u_\lambda)\d x \leq
2\int_{\Omega}\rho^0(D^2 u)+2\int_{\partial\Omega}|\gamma_1(u)-\bar
f_2|\d\H^1\,.
\]
By Theorem \ref{thm:relax2}, we may find  $U_\lambda\in H^{2}(\Omega)$ with
$U_\lambda=f_\lambda$ on $\partial\Omega$ such
that it satisfies
\[
\begin{split}
\lambda^{-1/2}\int_{\Omega}
F_\lambda(\nabla^2 U_\lambda)\d x
&\leq \int_{\Omega}
  G_\lambda(\nabla^2
  u_\lambda)\d x+\frac{1}{\lambda}\\
  \|U_\lambda-u_\lambda\|_{W^{1,2}(\Omega)}&\leq \frac{1}{\lambda}\,.
\end{split}
\]
This implies that $U_\lambda\to u$ in $W^{1,1}(\Omega)$. Also, we have
\[
\limsup_{\lambda\to\infty}
\F_{\bar f_\lambda,\lambda}(U_\lambda)\leq 2\int_{\Omega}\rho^0(D^2 u)+2\int_{\partial\Omega}|\gamma_1(u)-\bar
f_2|\d\H^1\,.
\]
By the compactness
part, it follows that $U_\lambda \to u$ weakly * in $BH(\Omega)$. 
This proves that $U_\lambda$ is the required recovery sequence.
\end{proof}


\section{Proof of the lower bound}
\label{sec:proof-lower-bound}
\begin{lemma}
\label{lem:glamtri}
  Let $\Omega\subset \R^2$ be open  and bounded, $\varphi\in C^0(\Omega;\R^2)$ and $w_\lambda\to 0$ in $L^1(\Omega)$ with
  $\|\nabla w_\lambda\|_{L^1}\leq C$. Then
  $G_\lambda(\varphi\otimes w_\lambda)\to 0$ in $L^1$.
\end{lemma}
\begin{proof}
  By the Poincar\'e inequality,
\[
\left\|w_\lambda-\left(\fint_{\Omega}w_\lambda(x)\d x\right)\right\|_{L^2}\leq C\|\nabla
w_{\lambda}\|_{L^1}\leq C\,.
\]
For $\lambda$ large enough we may assume $|\fint_{\Omega}w_\lambda(x)\d x|<1$,
and hence
$\|w_{\lambda}\|_{L^2}<C$. In particular, we have
\[
\L^2\left(\{x:|w_\lambda|>C^{-1}\sqrt{\lambda}\}\right)<\frac{C}{\lambda}\,.
\]
Now by $G_\lambda(\xi)=2\rho^0(\xi)-2\lambda^{-1/2}|\det \xi|\lesssim |\xi|$ for
$\rho^0(\xi)\leq\sqrt{\lambda}$ and $|\xi|\leq\rho^0(\xi)$, we have
\[
\begin{split}
  \int_{\Omega} G_{\lambda}(\varphi\otimes w_{\lambda})\d x\leq&
  C\left(\int_\Omega |w_\lambda|+\frac{|w_{\lambda}|^2}{\sqrt{\lambda}}\d
    x\right)+\lambda^{1/2} \L^2\left(\{x:|w_\lambda|>C^{-1}\sqrt{\lambda}\}\right)\\
    \to &0\quad\text{ as } \lambda\to \infty\,.
  \end{split}
\]
This proves the lemma.
\end{proof}

\begin{lemma}
\label{lem:affinebc}  
\begin{itemize}
\item[(i)] 
Let $\Omega\subset \R^2$ be open and bounded, $\xi_0\in \R^{2\times 2}$,
and  $w_\lambda\to 0$ in $L^1(\Omega)$ as $\lambda\to \infty$. Then
\[
\liminf_{\lambda\to\infty}\int_\Omega G_\lambda(\xi_0+\nabla w_\lambda)\d x\geq
2\L^2(\Omega)\rho^0(\xi_0)\,.
\]
\item[(ii)]
Let $\nu\in\R^2$ with $|\nu|=1$ and let $\tilde Q\subset \R^2$ be a cube such
that one of its sides is parallel to $\nu$, and let  $v\in L^1(\tilde Q;\R^2)$ such that
\[
v(y)=\psi(y\cdot \nu)\eta\quad\text{ for all }y\in\tilde Q
\]
for some $\eta\in\R^2$ and  $\psi\in BV((0,1))$. Furthermore let $v_j\to v\in
L^1(\tilde Q)$, and $\lambda_j\to\infty$. Then
\[
\liminf_{j\to\infty}\int_{\tilde Q}f_{\lambda_j}(\nabla v_j)\d x\geq
2\rho^0(Dv)(\tilde Q)\,.
\]
\end{itemize}
\end{lemma}

\begin{proof}[Proof of (i)]
We may assume $\lim_{\lambda\to\infty}\int_\Omega G_\lambda(\xi_0+\nabla
w_\lambda)=C<\infty$. We claim that for
$A$, $B\in \R^{2\times 2}$, we have
\begin{equation}
G_\lambda(A+B)\leq C  \left(G_{\lambda}(A)+G_\lambda(B)\right)\,.\label{eq:27}
\end{equation}
Indeed, by the sublinearity of $\rho^0$, we may assume
$\rho^0(A)\geq\rho^0(A+B)/2$ and $\rho^0(A)\geq\rho^0(B)$. Using the fact that
$|\xi|/2\leq \rho^0(\xi)\leq 2|\xi|$, we can make the following
case distinction:
\begin{itemize}
\item[Case 1:] If $\rho^0(A+B),\rho^0(A)> \sqrt{\lambda}$, then we have
\[G_\lambda(A)=\lambda^{1/2}+\lambda^{-1/2}|A|^2\geq
\lambda^{1/2}+\lambda^{-1/2}\left(\frac{\rho^0(A+B)}{4}\right)^2\geq C
G_\lambda(A+B)\]
\item[Case 2:] If $\rho^0(A)\leq \sqrt{\lambda}, \rho^0(A+B)\geq \lambda$, then
$\sqrt{\lambda}/2\leq \rho^0(A)\leq\rho^0(A+B)\leq
2\sqrt{\lambda}$. Additionally,  by \eqref{eq:2}, we have $\rho^0(A)\leq G_\lambda(A)$. This allows us to
estimate
\[
G_\lambda(A)\geq \rho^0(A)\geq \sqrt{\lambda}/2 \geq
C\left(\sqrt{\lambda}+\lambda^{-1/2}(2\lambda^{1/2})^2\right)\geq C G_\lambda(A+B)\,.
\]
\item[Case 3:] If $\rho^0(A),\rho^0(A+B)\leq\sqrt{\lambda}$, then again by
  \eqref{eq:2}, we have $\rho^0(A)\leq G_\lambda(A)$, and hence
\[
G_\lambda(A)\geq \rho^0(A)\geq \frac12 \left(\rho^0(A+B)-\lambda^{-1/2}|\det
  (A+B)|\right)=\frac14 G_\lambda(A+B)\,.
\]
\end{itemize}
This proves \eqref{eq:27}. Hence we have
\[
\int_\Omega G_\lambda(\nabla w_\lambda)\d x<C\,,
\]
and there exists a non-negative Radon measure $\mu$ such that
\[
G_\lambda(\nabla w_\lambda)\L^2\to \mu\text{ weakly * in the sense of measures.}
\]
Let $\Omega_k$ be an increasing sequence of subdomains s.t.
\[
\Omega_k\Subset\Omega, \quad\Omega=\cup_{k\in\N}\Omega_k
\]
and let $\varphi^k$ be smooth cutoff functions with $0\leq \varphi^k\leq 1$,
$\varphi^k|_{\Omega_k}=1$, $\varphi^k|_{\Omega\setminus\Omega_{k+1}}=0$. We set 
\[
w_\lambda^k= \varphi^k w_\lambda\,.
\] By the quasiconvexity of $G_\lambda$, we have
\[
\begin{split}
  \L^2(\Omega)G_\lambda(\xi_0)\leq& \int_\Omega G_\lambda(\xi_0+\nabla w_\lambda^k)\d
  x\\
  \leq &\L^2(\Omega\setminus\Omega_{k+1})G_\lambda(\xi_0)\d
  x+\int_{\Omega_{k+1}\setminus \Omega_k}G_\lambda(\xi_0+\nabla w_\lambda^k)\d
  x\\
&+\int_{\Omega_k}G_\lambda(\xi_0+\nabla w_\lambda)\d x\,.
\end{split}
\]
This implies 

\begin{equation}
\L^2(\Omega_{k+1})G_\lambda(\xi_0)\leq \int_{\Omega_{k+1}\setminus \Omega_k}G_\lambda(\xi_0+\nabla w_\lambda^k)\d
  x +\int_{\Omega_k}G_\lambda(\xi_0+\nabla w_\lambda)\d x\,.\label{eq:4}
  \end{equation}
We estimate 
\begin{equation}
\begin{split}
  \int_{\Omega_{k+1}\setminus \Omega_k}G_\lambda(\xi_0+\nabla w_\lambda^k)\d x&\lesssim
   \int_{\Omega_{k+1}\setminus
    \Omega_k}\left(G_\lambda(\xi_0)+G_\lambda(\nabla\varphi^k\otimes
  w_\lambda)+G_\lambda(\varphi^k\nabla w_\lambda)\right)\d x\\
   &\lesssim \L^2(\Omega_{k+1}\setminus\Omega_k)G_\lambda(\xi_0)+\int_{\Omega_{k+1}\setminus
    \Omega_k} G_\lambda(\nabla\varphi^k\otimes w_\lambda)\\
&\qquad+\int_\Omega
  G_\lambda(\nabla w_\lambda)(\varphi^{k+1}-\varphi^{k-1})\d x\,,
\end{split}\label{eq:9}
\end{equation}
where we have used \eqref{eq:27} in the first inequality. 
Subtracting the inequality \eqref{eq:4} from $\int_\Omega G_\lambda(\xi_0+\nabla
w_\lambda)\d x$, and additionally using \eqref{eq:9}, we obtain 
\begin{equation}
\begin{split}
  \int_\Omega G_\lambda&(\xi_0+\nabla w_\lambda)\d
  x-\L^2(\Omega_{k+1})G_\lambda(\xi_0)\\
  &\gtrsim\int_{\Omega\setminus\Omega_k}G_\lambda(\xi_0+\nabla w_\lambda)\d x
  -\L^2(\Omega_{k+1}\setminus\Omega_k)G_\lambda(\xi_0)\\
  &\qquad-\int_{\Omega_{k+1}\setminus \Omega_k} G_\lambda(\nabla\varphi^k\otimes
  w_\lambda)\d x -\int_\Omega G_\lambda(\nabla
  w_\lambda)(\varphi^{k+1}-\varphi^{k-1})\d x\,.
\end{split}\label{eq:48}
\end{equation}
We have $\|\nabla
w_\lambda\|_{L^1}\lesssim \|G_\lambda(\nabla w_\lambda)\|_{L^1}\leq C$, and
hence by Lemma
\ref{lem:glamtri}, we have 
\begin{equation}
G_\lambda(\nabla\varphi^k\otimes w_\lambda)\to 0\quad\text{ in
$L^1$.}\label{eq:15}
\end{equation}
Sending $\lambda\to\infty$ in \eqref{eq:48} and using \eqref{eq:15} and the
non-negativity of $G_\lambda$,  we obtain
\[
\begin{split}
  \lim_{\lambda\to\infty}\int_\Omega G_\lambda&(\xi_0+\nabla w_\lambda)\d
  x-\L^2(\Omega_{k+1})2\rho^0(\xi_0)\\
&\gtrsim -\L^2(\Omega_{k+1}\setminus
  \Omega_k)\rho^0(\xi_0)-\int_\Omega (\varphi^{k+1}-\varphi^{k-1})\d \mu\,.
\end{split}
\]
Summing up from $k=2$ to $k=l$ and dividing by $l-1$, we get
\[
\begin{split}
  \lim_{\lambda\to \infty}\int_\Omega &G_\lambda(\xi_0+\nabla w_\lambda)\d
  x-\left(\frac{1}{l-1}\sum_{k=2}^l\L^2(\Omega_{k+1})\right)2\rho^0(\xi_0)\\
  \gtrsim&-\frac{1}{l-1}\L^2(\Omega_{l}\setminus \Omega_2)\rho^0(\xi_0)-\frac{1}{l-1}\int_\Omega
  (\varphi^{l+1}+\varphi^{l}-\varphi^2-\varphi^1)\d\mu\\
  \gtrsim& -\frac{1}{l-1}\left(\L^2(\Omega)\rho^0(\xi_0)+|\mu|(\Omega)\right)\,.
  \end{split}
\]
Sending $l\to\infty$, we obtain the desired result.
\end{proof}
\begin{proof}[Proof of (ii)]
  The proof of Lemma 4.3 (ii) in \cite{ambrosio1992relaxation} can be copied
  word by word; except for the last step, where instead of Lemma 4.3 (i) in that
  reference we
  use part (i) of the present lemma.
\end{proof}

\begin{proof}[Proof of the lower bound in Proposition \ref{prop:gammaairy}]
We may assume $\F_{\bar f}(u)<\infty$ and after choosing an appropriate subsequence, we
may also assume $\lim_{\lambda\to
    \infty}\F_{\bar f_\lambda,\lambda}(u_\lambda)=\F_{\bar f}(u)$. 

Let $(\gamma_0,\gamma_1)^{-1}$ be the right inverse of the trace operator from Theorem
\ref{thm:traceop}, and let  $E$ be the extension operator $L^1(\Omega)\to L^1(\R^2)$ from Theorem
\ref{thm:steinext}.
 Letting $f_\lambda:=E\circ (\gamma_0,\gamma_1)^{-1}\bar f_\lambda$,  
we have  that $\tilde u_\lambda:\R^2\to\R$ defined by
\[
\tilde u_\lambda(x):=\begin{cases}u_\lambda(x)&\text{ if }x\in\Omega\\
f_\lambda(x)&\text{ else }\end{cases}
\]
satisfies $\tilde u_\lambda\in W^{2,2}(\R^2)$. Setting
$A(f_\lambda):=\{x\in\R^2:\rho^0(\nabla^2f_\lambda(x))>\lambda^{1/2}\}$ and 
$\tilde \Omega_\e:=\{x\in\R^2\setminus \Omega:\dist(x,\Omega)<\e\}$,  we
have 

\begin{equation}
\begin{split}
\label{eq:28}
  \int_{\tilde \Omega_\e}G_\lambda(\nabla^2\tilde u_\lambda)\d x &\lesssim
  \int_{\tilde \Omega_\e\setminus A(f_\lambda)} |\nabla^2 f_\lambda|\d x\\
&\quad +
  \int_{\tilde \Omega_\e\cap A(f_\lambda)} \frac{|\nabla^2 f_\lambda|^2}{\lambda^{1/2}}\d
    x+\L^2(\Omega_\e\cap A(f_\lambda)) \lambda^{1/2}\\
&\lesssim \|\nabla^2 f_\lambda\|_{L^{1}(\tilde\Omega_\e)}+\int_{\R^2} \frac{|\nabla^2 f_\lambda|^2}{\lambda^{1/2}}\d
    x
  \end{split}
\end{equation}
By assumption \eqref{eq:25} and the continuity of $E\circ
(\gamma_0,\gamma_1)^{-1}$ we have
\begin{equation}
\int_{\R^2} \frac{|\nabla^2 f_\lambda|^2}{\lambda^{1/2}}\d
    x\to 0\quad \text{ and }\quad f_{\lambda}\wto E\circ
    (\gamma_0,\gamma_1)^{-1} \bar f \quad\text{ in }W^{2,1}(\R^2) \,.\label{eq:32}
\end{equation}
The latter implies in particular that $|\nabla^2f_\lambda|$ is equiintegrable,
and hence  we obtain from \eqref{eq:28} and \eqref{eq:32} that
\begin{equation}
\lim_{\e\to 0}\lim_{\lambda\to\infty}  \int_{\tilde \Omega_\e}G_\lambda(\nabla^2\tilde u_\lambda)\d x =0\label{eq:34}\,.
\end{equation}
From now on we write $v_\lambda=\nabla \tilde u_\lambda$. By assumption, the sequence $\F_{\bar
    f_\lambda,\lambda}(u_\lambda)$ is bounded, and hence
\begin{equation}
\label{eq:18}
\int_{\R^2}G_\lambda(\nabla v_\lambda)\d x=\F_{\bar
    f_\lambda,\lambda}(u_\lambda)+\int_{\Omega\setminus\R^2}G_\lambda(\nabla^2\tilde u_\lambda)\d x \leq C
\end{equation}
by \eqref{eq:28} (with $\e=\infty$). 
Also, by $|\xi|\lesssim G_\lambda(\xi)$ for all $\xi\in \Rsym$, we have
\begin{equation}
\int_{\R^2}|\nabla v_\lambda|\d x\leq C\,.\label{eq:19}
\end{equation}
By \eqref{eq:18} and \eqref{eq:19} and the compactness theorems for $BV$
functions and Radon
measures respectively, there exists a subsequence of $v_\lambda$
(no relabeling), a measure $\mu\in \M(\R^2)$ and  $
v\in BV(\R^2;\R^2)$ with $v=\nabla u$ on $\Omega$, such that
\[
\begin{split}
  v_\lambda &\to  v \quad\text{ weakly * in }BV(\R^2;\R^2)\\
  G(\nabla v_\lambda)\L^2 &\to \mu \quad\text{ weakly * in }\M(\R^2)\,.
\end{split}
\]
By  Theorem \ref{thm:RN} there exists an $\L^2$ measurable function $\xi$ and a
$|D^sv|$ measurable function $\zeta$ such that
\[
\mu=\xi\, \L^2+\zeta\, |D^sv|\,.
\]
We are going to show 
\begin{align}
\label{eq:7}
\xi(x_0)&\geq 2\rho^0(\nabla v(x_0))&\, &\text{ for } \L^2-\text{a.e. }x_0\in \Omega\\
\label{eq:8}
\zeta(x_0)&\geq 2&\, &\text{ for } |D^sv|-\text{a.e. }x_0\in \overline\Omega\,.
\end{align}
We claim that this implies \eqref{eq:26}. Indeed, recall that the right hand side in
\eqref{eq:26} reads 
\[
2\rho^0(Dv)(\Omega)+2\int_{\partial\Omega}|v\cdot n-\bar f_2|\d\H^1
=2\rho^0(D v)(\overline\Omega)\,,
\]
see equations \eqref{eq:36} and \eqref{eq:46} in the proof of the upper bound.
Let
$\Omega_\e:=\{x\in\R^2:\dist(x,\Omega)<\e\}$. By \eqref{eq:34}, the left hand
side of \eqref{eq:26} is equal to
\[
\begin{split}
  \lim_{\e\to 0}\lim_{\lambda\to\infty} \int_{\Omega_\e}G_\lambda(\nabla
  v_\lambda)\d x
  &=\lim_{\e\to 0} \mu(\Omega_\e)\\
  &=\mu(\overline{\Omega})\,.
\end{split}
\]
Now we see that \eqref{eq:7} and \eqref{eq:8} imply $\mu(\overline{\Omega})\geq
2\rho^0(Dv)(\overline{\Omega})$, and hence prove the lower bound part.
It remains to
show \eqref{eq:7} and \eqref{eq:8}.

First we  prove \eqref{eq:7}. Let $Q(x_0,\e):=x_0+[-\e/2,\e/2]^2$. 
For $\L^2$-almost every $x_0$, we may choose a sequence $(\e_j)_{j\in\N}$ converging to
zero, such that
$\mu(\partial Q(x_0,\e_j))=0$ for every $j\in \N$. When we write $\e\to 0$ in the
sequel, we actually mean the limit $j\to\infty$ for such a sequence. For every
$j$, we  have
\[
\lim_{\lambda\to\infty}\int_{Q(x_0,\e_j)} G_\lambda(\nabla v_\lambda)\d x\to \mu(Q(x_0,\e_j))\,.
\]
Note that by Theorem \ref{thm:RN} we have
\begin{equation}
\begin{split}
  \xi(x_0)=&\lim_{\e\to 0} \frac{\mu(Q({x_0,\e}))}{\L^2(Q(x_0,\e))}\\
  =&\lim_{\e\to\infty}\lim_{\lambda\to \infty}\fint_{Q(x_0,\e)}G_\lambda(\nabla
  v_\lambda)\d x\,.
\end{split}\label{eq:12}
\end{equation}
For $\e$ small enough, define $w_{\lambda,\e}:Q\to \R^2$ by
\[
w_{\lambda,\e}(x)=\e^{-1}\left(v_\lambda(x_0+\e x)-v(x_0)\right)\,.
\]
Furthermore let $w_0(x)=\nabla v(x_0)\cdot x$. Using a change of variables, the
Cauchy-Schwarz inequality and \eqref{eq:47}, we have
\begin{equation}
\begin{split}
  \lim_{\e\to 0}\lim_{\lambda\to\infty} \|w_{\lambda,\e}-w_0\|_{L^1(Q)}&=
\lim_{\e\to 0} \frac{1}{\e}\int_{Q}|v(x_0+\e x)-v(x_0)-\nabla v(x_0)\cdot
  \e x|\d x\\
 &= \lim_{\e\to 0} \frac{1}{\e^3}\int_{Q(x_0,\e)}|v(x)-v(x_0)-\nabla v(x_0)\cdot
  (x-x_0)|\d x\\
&\leq \lim_{\e\to 0}\frac{1}{\e^2}\left(\int_{Q(x_0,\e)}|v(x)-v(x_0)-\nabla v(x_0)\cdot
  (x-x_0)|^2\d x\right)^{1/2}\\
  &=0\,.
\end{split}\label{eq:11}
\end{equation}

By \eqref{eq:12} and \eqref{eq:11}, it is possible to choose a sequence
$\lambda_j\to \infty$ and a subsequence $\e_j\to 0$ (no relabeling) such that with $w_j:=w_{\lambda_j,\e_j}$
\[
\begin{split}
  \lim_{j\to \infty}\|w_j-w_0\|&=0\\
  \lim_{j\to \infty}\fint_{Q(x_0,\e_j)}G_{\lambda_j}(\nabla v_{\lambda_j})\d
  x&=\xi(x_0)\,.
\end{split}
\]
Noting $\fint_{Q(x_0,\e_j)}G_{\lambda_j}(\nabla v_{\lambda_j})\d x=\int_Q
G_{\lambda_j}(\nabla w_j)\d x$, we obtain from Lemma \ref{lem:affinebc} that
\[
\begin{split}
  \xi(x_0)=\lim_{j\to\infty}\int_Q  G_{\lambda_j}(\nabla w_j)\d x
  \geq 2\rho^0(\nabla w_0)
  = 2\rho^0(\nabla v(x_0))\,.
\end{split}
\]
This proves \eqref{eq:7}.

We turn to the proof of \eqref{eq:8}. 
By Theorem \ref{thm:RN} we have that for $|D^sv|$-almost every $x_0\in\overline\Omega$,
there exist $\eta,\nu\in \R^2$ with $|\eta|=|\nu|=1$ such that for any open bounded convex
set $K$ containing the origin, we have
\[
\lim_{\e\to 0}\frac{Dv(x_0+\e K)}{|Dv|(x_0+\e K)}=\eta\otimes\nu\,.
\]
Let $Q^\nu\subset\R^2$ denote the cube of sidelength one, with one axis parallel to $\nu$:
\[
Q^\nu:=\left\{x\in\R^2:|x\cdot\nu|<\frac12,\,|x\cdot\nu^\bot|<\frac12\right\}\,.
\]
Furthermore, let $Q^\nu(x_0,\e)=x_0+\e Q^\nu$.
From now on, let the limit $\e\to 0$ be understood only to involve a sequence
$(\e_j)_{j\in\N}$ such that $\mu(\partial Q^\nu(x_0,\e_j))=0$ for all $j\in\N$.
By the definition of $\zeta$, we have

\begin{equation}
\begin{split}
  \zeta(x_0)=&\lim_{\e\to 0} \frac{\mu(Q^\nu(x_0,\e))}{|Dv|(Q^\nu(x_0,\e))}\\
    =&\lim_{\e\to 0}\lim_{\lambda\to\infty}
    \frac{1}{|Dv|(Q^\nu(x_0,\e))}\int_{Q^\nu(x_0,\e)}G_{\lambda}(\nabla
      v_\lambda)\d x\,.
    \end{split}\label{eq:10}
      \end{equation}
We define
\[
\begin{split}
  w_{\lambda,\e}(x)=&\frac{\e}{|Dv|(Q^\nu(x_0,\e))}\left(v_\lambda(x_0+\e
    x)-\fint_{Q^\nu(x_0,\e)}v_\lambda(x') \d x'\right)\\
    w_{\e}(x)=&\frac{\e}{|Dv|(Q^\nu(x_0,\e))}\left(v(x_0+\e
      x)-\fint_{Q^\nu(x_0,\e)}v(x') \d x'\right)\,.
    \end{split}
\]
Let $\sigma\in (0,1)$.
By Theorem \ref{thm:BVblow} there exists 
$\tilde w\in BV_{\loc}(\R)$ such that with
\[
w_0(x):=\tilde w (x\cdot \nu)\eta
\]
we have  $|Dw_0|(Q_\nu)\geq \sigma^2$ and
\[
\lim_{\e\to 0}\|w_\e-w_0\|_{L^1(Q^\nu)}=0\,.
\]
By the convergence $v_\lambda\to v$ in $L^1(\Omega)$, we have
$\lim_{\lambda\to\infty}\|w_{\lambda,\e}-w_\e\|_{L^1(Q^\nu)} =0$, and hence
\begin{equation}
 \lim_{\e\to 0}\lim_{\lambda\to 0} \|w_{\lambda,\e}-w_0\|_{L^1(Q^\nu)}  =0\,.
\label{eq:6}
\end{equation}
By \eqref{eq:10} and \eqref{eq:6}, we can choose a subsequence $\lambda_j$ and
a sequence  $\e_j$ such that  
\begin{equation}
  \label{eq:13}
  \begin{split}
    \lambda_j\frac{|Dv|(Q_{x_0,\e_j})^2}{\e_j^4}\to&\infty\\
    \lim_{j\to\infty
    }\frac{1}{|Dv|(Q_{x_0,\e_j})}\int_{Q_{x_0,\e_j}}G_{\lambda_j}(\nabla
    v_{\lambda_j})\d x=&\zeta(x_0)\\
    \lim_{j\to \infty}\|w_{\lambda_j,\e_j}-w_0\|_{L^1(Q)}=&0\,.
  \end{split}
\end{equation}
Now let $\tilde \e_j:= \frac{|Dv|(Q_{x_0,\e_j})}{\e_j^2}$. Then we have
$\lambda_j\tilde \e_j^2\to\infty$ and
\[
\frac{1}{|Dv|(Q_{x_0,\e_j})}\int_{Q_{x_0,\e_j}}G_{\lambda_j}(\nabla
  v_{\lambda_j})\d x=\fint G_{\lambda_j\tilde \e_j^2}(\nabla
  w_{\lambda_j,\e_j})\d x\,.
\]
Hence using Lemma \ref{lem:affinebc} (ii) it follows 
\[
\zeta(x_0)=\lim_{j\to\infty}\fint G_{\lambda_j\tilde \e_j^2}(\nabla
  w_{\lambda_j,\e_j})\d x\geq 2|Dw_0|(Q^\nu)\geq 2\sigma^2\,.
\]
Sending $\sigma\uparrow 1$  proves \eqref{eq:8} and completes the proof of the
lower bound.
\end{proof}


\appendix

\section{Proof of Theorem \ref{thm:relax2}}
\label{sec:proof-theor-refthm:r}
In analogy to the proof of the relaxation leading to 1-quasiconvex integrands,
we first need to prove an approximation lemma. This is slightly more complicated
here than in the case of 1-quasiconvexity, since we cannot  use
the approximation by finite elements here, which is possible there (see
\cite{dacorogna1982quasiconvexity}). Instead, we are going to need Whitney's extension theorem, that
we quote here in a version that can be found in Stein's book \cite{stein2016singular}. Let
$\Omega\subset\R^n$. 
Let the Greek letters $\alpha,\beta,\gamma\in \N_0^n$ denote multiindices. We
will write
$|\alpha|=\sum_i\alpha_i$,  $\alpha!=\prod_i\alpha_i!$, and
$\nabla^\alpha=\partial_1^{\alpha_1}\dots\partial_n^{\alpha_n}$. Furthermore,
for $x\in\R^n$, we write $x^\alpha=x_1^{\alpha_1}\dots x_n^{\alpha_n}$.
We shall say that a function $f:\Omega\to\R$ belongs to $
\mathrm{Lip}^{(k,1)}(\Omega)$ if there exists a collection of real-valued functions
$\{f^{(\alpha)}:|\alpha|\leq k\}$ and a constant $M>0$ such that
\[
f^{(\alpha)}(y)=\sum_{\beta:|\alpha+\beta|\leq k}\frac{f^{(\alpha+\beta)}(x)}{\beta!} (x-y)^{\beta} +R_\alpha(x,y)
\]
and 
\[
|f^{(\alpha)}(y)|\leq M \text{ and } |R_\alpha(x,y)|\leq M|x-y|^{k+1-|\alpha|}
\]
for all $x,y\in\Omega$ and all multiindices $\alpha$ with $|\alpha|\leq
k$. The set $\mathrm{Lip}^{(k,1)}(\Omega)$ is a normed space, where the norm of $f$ is given by
the smallest constant $M$ such that the above relations hold true.
\begin{theorem}[Theorem 4 in Chapter 6 of \cite{stein2016singular}]
\label{thm:stein}
 Let $k$ be a non-negative integer and let $\Omega\subset\R^n$ be closed. Then there exists a continuous extension
 operator   $\mathrm{Lip}^{(k,1)}(\Omega)\to \mathrm{Lip}^{(k,1)}(\R^n)$. The
 norm of this mapping has a bound that is independent from $\Omega$.
\end{theorem}

For a closed set $\Gamma\subset \R^n$, let
$d_{\Gamma}\in C^\infty(\R^n;[0,\infty))$ denote a regularized distance function, that
is, a function with the property that there exists a constant $C>0$ such that
\[
\begin{split}
  C^{-1}\dist(x,\Gamma)&\leq d_{\Gamma}(x)\leq C\dist(x,\Gamma)\quad \text{ for all
  }x\in \R^N\,.\\
  |\nabla^\alpha d_{\Gamma}(x)|&\leq C |d_{\Gamma}(x)|^{1-|\alpha|}\quad\text{ for all multiindices } \alpha\,.
\end{split}
\]
Such a regularized distance function exists by Theorem 2 of Chapter 6 in
\cite{stein2016singular}. We will use it to construct suitable cutoff functions
in Lemma \ref{lem:prepapprox} below. This lemma is a preparation for 
the approximation lemma, Lemma~\ref{lem:approx} below. 

\medskip

Let $ \phi\in C^\infty([0,\infty))$ such that $\phi(t)=0$ for
$t\leq\frac12$ and $\phi(t)=1$ for $t\geq 1$.

\begin{lemma}
\label{lem:prepapprox}
Let $p\geq 1$, $U\subset[0,1]^2$, 
$(u_\e)_{\e>0}$ a sequence in $W^{2,p}(U)$ that converges strongly to $u\in
W^{2,p}(U)$ for $\e\to 0$, 
and let
\[\Gamma\subset\left([0,1]\times\{0\}\cup\{0\}\times[0,1]\right)\cap \partial U\]
be a closed set of
positive $\H^1$ measure,
and
$u_0\in W^{2,p}(U)$ such that $u=u_0$ and $\nabla u=\nabla u_0$ on
$\Gamma$. Furthermore let 
\[
\tilde u_\e(x)=\phi\left(\frac{d_{\Gamma}(x)}{\e}\right)
u_\e(x)+\left(1-\phi\left(\frac{d_{\Gamma}(x)}{\e}\right) \right)u_0(x)\,.
\]
Then $\tilde u_\e=u_0$ and $\nabla \tilde u_\e=\nabla u_0$ on $\Gamma$,
and 
\[
\begin{split}
  \|\tilde u_\e- u\|_{W^{2,p}(U)}&\to 0\\
  \int_{U\cap\{x:d_\Gamma(x)\leq\e\}}\left(1+|\nabla^2 u_\e|^p+|\nabla u|^p\right)\d x&\to 0\,.
\end{split}
\]
\end{lemma}

\begin{proof}
The first claim is obvious. To show the second  claim, it suffices to show
that $\|\nabla^2(\tilde u_\e-u)\|_{L^p}\to 0$, since  $\tilde u_\e-u$ and its gradient vanish on $\Gamma$. This is a straightforward
computation, that estimates the integral over $|\nabla^2(\tilde u_\e-u)|^p$
in the ``bulk'' $U\cap\{x:d_\Gamma(x)>\e\}$  and in boundary layer
\[
U\cap\left\{x:d_\Gamma(x)\leq\e\right\}\subset (U\cap \{x:x_1<C\e\})\cup (U\cap
\{x:x_2<C\e\})\,,
\]
where the constant $C>0$ is chosen appropriately.
The contribution of the bulk vanishes in the limit
$\e\to 0$ due to the assumption. Writing $\tilde u_\e-u$ and its gradient as
integrals in $x_i$-direction in the set $U\cap \{x:x_i<C\e\}$ for
$i=1,2$ and using
Fubini's theorem, the proof that the contribution of the boundary layer vanishes
in the limit is a straightforward but lengthy computation that we omit here for
the sake of brevity. The third claim is
an immediate consequence of the $L^p$ integrability of $\nabla^2 u$ and the
strong convergence $\nabla^2u_\e\to\nabla^2 u$ in $L^p$.
\end{proof}

\begin{lemma}
\label{lem:approx}
Let $\Omega\subset\R^2$ satisfy Definition \ref{def:H1}, and let $p\in
[1,\infty)$.  Furthermore let $u\in \{u_0\}+W^{2,p}_0(\Omega)$ and $\delta>0$. Then there exists $w\in 
  \{u_0\}+W^{2,p}_0(\Omega)$ and $\Omega_w\subset\Omega$ such that $\Omega_w$ is
  the union of mutually disjoint closed cubes, $w$ is
  piecewise a polynomial of degree $k$ on $\Omega_w$, and furthermore
\[
\begin{split}
  \|u-w\|_{W^{2,p}(\Omega)}&<\delta\,,\\
\int_{\Omega\setminus\Omega_w}(1+|\nabla^2u|^p+|\nabla^2w|^p)\d x&<\delta\,.
\end{split}
\]
\end{lemma}
\begin{proof}
  Let $v:=u-u_0\in W^{2,p}_0(\Omega)$. We may extend $u_0$  to $\R^2$ such
    that 
\[
\begin{split}
  \|u_0\|_{W^{2,p}(\R^2)}&\lesssim \|u_0\|_{W^{2,p}(\Omega)}\,.
        \end{split}
\]
Also, $v$ may be understood as an element of $W^{2,p}(\R^2)$ by a trivial extension.
Let $\{\Omega_i:i=1,\dots,P\}$ be an open cover of $\overline \Omega$, let $\{\psi_i:i=1,\dots,P\}$ be a
subordinate partition of unity, and let $\xi_i:\Omega_i\to \xi_i(\Omega_i)=:\tilde \Omega_i$,
$i=1,\dots,P$, be a set of $C^2$ diffeomorphisms such that $\xi_i(\Omega\cap
\Omega_i)\subset (0,1)^2$ and $\Gamma_i:=\xi_i(\Omega_i\cap\partial \Omega)\subset
[0,1]\times \{0\}\cup\{0\}\times [0,1]$.
(Such $\xi_i$ exist by the assumption on $\Omega$.)

\medskip

Let $\varphi\in C_c^\infty(\R^n)$ with $\int\varphi\d x =1$, and
$\varphi_\e=\e^{-n}\varphi(\cdot/\e)$. 
For every $i$, we have that
\[
\|(u- (\varphi_\e* u))\circ\xi_i^{-1}\|_{W^{2,p}(\tilde \Omega_i)}\to 0.
\]
By the previous lemma, there exists $\hat u^{(i)}_\e\in {W^{2,p}(\tilde \Omega_i)}$
such that $\hat u^{(i)}_\e=(\varphi_\e * u)\circ \xi_i^{-1}$ on $\tilde
\Omega_i\cap\{x:d_{\Gamma_i(x)}\geq \e\}$,
\[
 \hat u^{(i)}_\e= u\circ\xi_i^{-1}, \quad \nabla \hat u^{(i)}_\e= \nabla
 u\circ\xi_i^{-1} \quad\text{ on } [0,1]\times\{0\}\cap \tilde \Omega_i\,,
\]
and additionally,
\[
\begin{split}
  \|\hat u^{(i)}_\e- u\circ \xi_i^{-1}\|_{W^{2,p}(U)}&\to 0\\
  \int_{\tilde \Omega_i\cap\{x:d_{\Gamma_i(x)}<\e\}}1+|\nabla^2 \hat u_\e^{(i)}|^p+|\nabla u\circ\xi_i^{-1}|^p\d x&\to 0\,.
\end{split}
\]
Setting $u_\e^{(i)}:=\hat u^{(i)}_\e\circ \xi_i$, we obviously have that
$u_\e^{(i)}=\varphi_\e\circ u$ on $\Omega_i\setminus \xi_i^{-1}(\tilde
\Omega_i\cap\{x:d_{\Gamma_i(x)}< \e\})$, and
\[
\begin{split}
 u^{(i)}_\e= u, \quad \nabla  u^{(i)}_\e&= \nabla
 u \quad\text{ on } \Omega_i\cap\partial \Omega\,,\\
  \| u^{(i)}_\e- u\|_{W^{2,p}(\Omega_i)}&\to 0\\
  \int_{ \xi_i^{-1}(\tilde \Omega_i\cap\{x:d_{\Gamma_i(x)}<\e\})}1+|\nabla^2  u_\e^{(i)}|^p+|\nabla u|^p\d x&\to 0\,.
\end{split}
\]
Setting 
\[
\begin{split}
  \Omega_\e&:=\Omega\setminus \bigcup_{i=1}^N \xi_i^{-1}(\tilde \Omega_i\cap\{x:d_{\Gamma_i(x)}<\e\})\\
  u_\e&:=\sum_i \psi_i u_\e^{(i)}\,,
\end{split}
\]
we have $u_\e=\varphi_\e* u$ on $\Omega_\e$ and 
\[
\begin{split}
 u_\e= u, \quad \nabla  u_\e&= \nabla
 u \quad\text{ on } \partial \Omega\,,\\
  \| u_\e- u\|_{W^{2,p}(\Omega)}&\to 0\\
  \int_{\Omega\setminus \Omega_\e}1+|\nabla^2  u_\e|^p+|\nabla u|^p\d x&\to 0\,.
\end{split}
\]
and hence we may fix $\e$ such that  
\[
\begin{split}
  \|u_\e-u\|_{W^{k,p}(\Omega)}&<\delta/2\\
  \int_{\Omega\setminus\Omega_\e}(1+ |\nabla^k u|^p+|\nabla^k u_\e|^p)\d
  x&<\delta/2\,.
\end{split}
\]
Note that $\Omega_\e$ is
closed and 
$u_\e\in C^\infty(\Omega_\e)$. 

Let $0<h\ll 1$ to be chosen later, and let $Q_i$,
$i=1,\dots,N$ be mutually disjoint closed cubes of
sidelength $h$ contained in $\Omega_\e$  such that 
\[
\begin{split}
  \dist(Q_i,Q_j)&\geq h^{4/3} \quad\text{ for all }i,j=1,\dots,N\,, i\neq j\\
  \dist(\partial \Omega_\e,Q_i)&\geq h^{4/3} \quad\text{ for all }i=1,\dots,N\\
  \L^n\left(\Omega_\e\setminus \bigcup_i Q_i\right)&\leq C h^{1/3}\,.
\end{split}
\]
This is always possible assuming that $h$ is small enough. Let $x_i$ denote the midpoint of $Q_i$. Then we define $\tilde u_\e$ to be the
Taylor polynomial of degree two at $x_i$ on each $Q_i$, 
\[
\tilde u_\e(y)= \sum_{|\alpha|\leq 2} \nabla^\alpha u_\e(x_i) \frac{(y-x_i)^\alpha}{\alpha!}\quad\text{ for }y\in Q_i\,.
\] 
Let $V=(\cup_i Q_i)\cup\partial\Omega_\e$. We claim that there exists an extension of $\tilde u_\e$ from $V$ to $\Omega_\e$
such that $\|\tilde u_\e\|_{W^{2,\infty}(\Omega_\e)}\lesssim
\|u_\e\|_{W^{2,\infty}({\Omega_\e})}$. In order to prove our claim, we invoke   Theorem
\ref{thm:stein} with $k=1$. We verify that $\tilde u_\e\in
\mathrm{Lip}^{(1,1)}(V)$: Firstly, we have for all $y\in V$ and all
multiindices $\alpha$ with $|\alpha|\leq 1$, 
\begin{equation}
\begin{split}
  |\nabla^\alpha \tilde u_\e(y)-\nabla^\alpha u_\e(y)|\leq C h^{2-|\alpha|}\,,
\end{split}\label{eq:33}
\end{equation}
where the constant $C$ depends on $\|u_\e\|_{C^{3}(\Omega_\e)}$.
Furthermore, we have $u_\e\in
\mathrm{Lip}^{(1,1)}(\Omega_\e)=C^{1,1}(\Omega_\e)$, namely there exists a
constant $0<M_1\lesssim \|u_\e\|_{C^{1,1}(\Omega_\e)}$ such that
for all $x,y\in {\Omega_\e}$ and all multiindices $\alpha$ with $|\alpha|\leq 1$ we have
\begin{equation}
\begin{split}
 \nabla^\alpha u_\e(y)&=\sum_{|\alpha+\beta|\leq 1}\nabla^{\alpha+\beta} u_\e(x)\frac{(y-x)^\beta}{\beta!}+R_\alpha(x,y)
\end{split}\label{eq:23}
\end{equation}
with 
\[
|\nabla^\alpha u_\e(x)|<M_1\,,\quad
|R_\alpha(x,y)|<M_1|x-y|^{2-|\alpha|}\,.
\]
Now let $x,y\in Q_i$. Then we have for all multiindices $\alpha$ with
$|\alpha|\leq 1$, 
\begin{equation}
\begin{split}
  \nabla^\alpha \tilde u_\e(y)&=\sum_{|\alpha+\beta|\leq 2}\nabla^\alpha u_\e(x)\frac{(y-x)^{\beta}}{\beta!}\,.
\end{split}
\label{eq:43}
\end{equation}
Next let $x\in Q_i,y\in Q_j$ with $i\neq j$, or $x\in\partial\Omega$, $y\in
Q_i$. In this case, we have $|x-y|\geq h^{4/3}$. By inserting \eqref{eq:33} in
\eqref{eq:23}, we obtain
\begin{equation}
\begin{split}
 \nabla^\alpha \tilde u_\e(y)&= \sum_{|\alpha+\beta|\leq
   1}\nabla^{\alpha+\beta}\tilde
 u_\e(x)\frac{(y-x)^{\beta}}{\beta!}+O\left(h^{3-|\alpha|}\right)+
 R_\alpha(x,y)\\
&=  \sum_{|\alpha+\beta|\leq
   1}\nabla^{\alpha+\beta}\tilde
 u_\e(x)\frac{(y-x)^{\beta}}{\beta!}+
 \tilde R_\alpha(x,y)\,,
\end{split}
\label{eq:41}
\end{equation}
where we have introduced 
$\tilde R_\alpha=R_\alpha(x,y)+O(h^{3-|\alpha|})$, which by $|x-y|\geq
h^{4/3}$ implies
\[
|\tilde R_\alpha(x,y)-R_\alpha(x,y)|=O(h^{3-|\alpha|})= o(1)|x-y|^{2-|\alpha|}\,,
\]
where the last estimate holds since for all multiindices $\alpha$ with
$|\alpha|\leq 1$, we have
\[
\frac{4}{3}<\frac{3-|\alpha|}{2-|\alpha|}\,.
\]
Summarizing \eqref{eq:43} and \eqref{eq:41}, we have proved that $\tilde u_\e\in \mathrm{Lip}^{(1,1)}(V)$, with 
\[
|\nabla^\alpha\tilde u_\e(x)|<M_2\,,\quad
|\tilde R_\alpha(x,y)|<M_2|x-y|^{2-|\alpha|}\,,
\]
where
\[
M_2\leq M_1+o(1)\lesssim \|u_\e\|_{C^{1,1}({\Omega_\e})}\lesssim
\|u_\e\|_{W^{2,\infty}({\Omega_\e})}\,.
\]
By Theorem \ref{thm:stein}, there exists an extension of $\tilde u_\e$ to
$\mathrm{Lip}^{(1,1)}({\Omega_\e})$, with 
\[
\|\tilde u_\e\|_{W^{2,\infty}({\Omega_\e})}\lesssim \|u_\e\|_{W^{2,\infty}({\Omega_\e})}\,.
\]
Comparing the extension with $u_\e$, we have the estimates
\begin{equation}
\begin{split}
   \int_{\Omega_\e\setminus V}|\nabla^2 \tilde u_\e-\nabla^2 u_\e|^p\d x
  &\lesssim h^{1/3}\|
  u_\e\|_{W^{2,\infty}(\Omega_\e)}^p\,,\\
  \int_{V}|\nabla^2 \tilde u_\e-\nabla^2 u_\e|^p\d x 
&\lesssim  \L^2(V)h^p\,.
\end{split}\label{eq:54}
\end{equation}
Choosing $h$ small enough, we have
\[
\int_{{\Omega_\e}\setminus V}(1+|\nabla^k u|^p)\d x<\delta/3\,,
\]
and 
$\|\tilde u_\e-u_\e\|_{W^{2,p}(\Omega_\e)}<\delta/2$.
We claim that the function $w\in W^{2,p}(\Omega)$, defined by
\[
w(x)=\begin{cases}\tilde u_\e&\text{ if }x\in\Omega_\e\\
u_\e&\text{ else }\end{cases}
\]
satisfies all the properties that are stated in the lemma.
Indeed, we have $w=u_0$ and $\nabla w=\nabla u_0$ on $\partial\Omega$,   $w$ is a
polynomial of degree $2$ on  on $\Omega_w:=\cup_i Q_i$, and 
\[
\begin{split}
  \|w-u\|_{W^{2,p}(\Omega)}&\leq \|u-u_\e\|_{W^{2,p}(\Omega)}+\|u_\e-\tilde u_\e\|_{W^{2,p}(\Omega)}<\delta\,,\\
\int_{\Omega\setminus\Omega_w}(1+|\nabla^2 u|^p+|\nabla^2 w|^p)\d x&\leq
2\delta/3+Ch^{1/3}\|
  u_\e\|_{W^{2,\infty}(\Omega_\e)}^p<\delta\,.
  \end{split}
\]
This proves the lemma.
\end{proof}

\begin{proof}[Proof of Theorem \ref{thm:relax2}]
Let $u\in u_0+W^{2,p}_0({\Omega})$. By Lemma \ref{lem:approx}, there exists
$w\in u_0+W^{2,p}_0({\Omega})$ and ${\Omega}_w\subset{\Omega}$ such that ${\Omega}_w$ is
a union of mutually disjoint closed cubes, ${\Omega}_w=\cup_{i=1}^NQ_i$, $w$ is
  piecewise a polynomial of degree $2$ on ${\Omega}_w$, and furthermore
\[
\begin{split}
  \|u-w\|_{W^{2,p}({\Omega})}&<\e/2\,,\\
\int_{{\Omega}\setminus{\Omega}_w}1+|\nabla^2 w|^p+|\nabla^2 u|^p\d x&<\e/2\,.\end{split}
\]
For $i \in \{1,\dots,N\}$, 
choose $\tilde \xi_i\in W_0^{2,\infty}([-1/2,1/2]^2;\R^m)$ such that
\[
\int_{[-1/2,1/2]^2} f(\nabla^2 w(x_i)+\nabla^2\tilde \xi_i(x))\d x< Q_2f(\nabla^2 w(x_i))+\frac{\e}{N\L^2(Q_i)}\,,
\]
where $x_i$ denotes the center of the cube $Q_i$.
We identify $\tilde \xi_i$ with its periodic extension to $\R^2$. Let $d_i$
denote the sidelength of the cube $Q_i$, and
let $M_i\in\N$ to be chosen later. We define $\xi_i\in W^{2,p}_0(Q_i;\R^m)$ by
\[
\xi_i(x)=\left(\frac{d_i}{M_i}\right)^{2}\tilde\xi_i\left(\frac{M_i(x-x_i)}{d_i}\right)\,.
\]
Then we have
\[
\begin{split}
  \int_{Q_i}f(\nabla^2 w(x_i)+\nabla^2\xi_i)\d x&=\L^2(Q_i)\int_{[0,1]^2} f(\nabla^2 w(x_i)+\nabla^2\tilde \xi_i(x))\d x\\
&  < \L^2(Q_i)\,Q_2f(\nabla^2 w(x_i))+\frac{\e}{N}\\
&= \int_{Q_i}Q_2f(\nabla^2 w(x))\d x +\frac{\e}{N}\,.
\end{split}
\]
Choosing $M_i$ large enough, we may assume
\[
\|\xi_i\|_{W^{1,p}(Q_i;\R^m)}^p<\frac{(\e/2)^p}{N}\,.
\]
Now the function 
\[
v(x)=\begin{cases} w+\xi_i&\text{ on }Q_i\\ w&\text{ on }{\Omega}\setminus
    {\Omega}_w\end{cases}
\]
has all the required properties.
\end{proof}

\section{Proof of Theorem \ref{thm:Q2F}}
\label{sec:proof-theor-refthm:q}

For the convenience of the reader, we repeat the statement. We set
\[
\bar G_\lambda(\sigma):=\begin{cases}2\sqrt{\lambda} \rho^0(\sigma)-2|\det \sigma| & \text{ if }\rho^0(\sigma)\leq
  \sqrt{\lambda}\\
 |\sigma|^2+\lambda & \text{ else,}\end{cases}
\]
and the theorem we want to prove is

\begin{theorem}
\label{thm:appthm}
  We have
\[
Q_2F_\lambda(\sigma)=\bar G_\lambda(\sigma)\,.
\]
\end{theorem}

For the proof, we will need to carry out proofs of statements whose analogues
for first gradients are well known. We closely follow the proofs in
\cite{MR2361288}, adapting them to the current situation.
\newcommand{\Rnsym}{\R^{n\times n}_{\mathrm{sym}}}
\begin{definition}
  Let $f:\Rnsym\to \R$. We say that $f$ is symmetric rank one convex if
\[
f(t\xi_1+(1-t)\xi_2)\leq t f(\xi_1)+(1-t) f(\xi_2) 
\]
for all $t\in[0,1]$, and for all $\xi_1,\xi_2\in\Rnsym$ such that $\xi_1-\xi_2=\alpha\eta\otimes \eta$ for
some $\alpha\in \R$, $\eta\in\R^n$.

Furthermore, for $f:\Rnsym\to \R$, we set
\[
\Rs f(\xi):=\sup\{g(\xi):g\leq f \text{ and } g \text{ is symmetric rank one
  convex}\}\,.
\]
\end{definition}


\begin{lemma}
\label{lem:auxconst1}
  Let  $\alpha,\beta\in\R$, $t\in [0,1]$, $\e>0$ and 
\[
u_t:[a,b]\to\R,\,\quad x\mapsto \frac12 (t\alpha+(1-t)\beta) x^2\,.
\]
Then there exist $I,J\subset[0,1]$ and $u:[0,1]\to\R$ such that $\overline
I\cup\overline J=[0,1]$, $I\cap J=\emptyset$, $|I|=t$, $|J|=(1-t)$ and \[
\begin{split}
  u(0)=u_t(0),\quad u'(0)=u'_t(0),\\
  u(1)=u_t(1),\quad u'(1)=u'_t(1),\\
  \|u-u_t\|_{L^\infty}+\|u'-u_t'\|_{L^\infty}<\e\\
  u''(x)= \alpha \quad\text{ for }x\in I\\
  u''(x)= \beta \quad\text{ for }x\in J\,.
\end{split}
\]
\end{lemma}

\begin{proof}
For $q\in[0,1-t]$, let $\varphi_{t,q}:[0,1]\to \R$ be defined by
\[
\varphi_{t,q}(x)=\begin{cases} (1-t)(\alpha-\beta)&\text{ if } q\leq x< q+t\\
-t(\alpha-\beta)&\text{ else. }\end{cases}
\]
Note that $\int_0^1\d s\varphi_{t,q}(s)=0$ independently of $q$. In fact, we may
choose $q$ such that we also have
\[
\int_0^1\d s\int_0^s \d \tilde s \varphi_{t,q}(\tilde s)=0\,.
\]
This choice of $q$ shall be fixed from now on. We extend $\varphi_{t,q}$ periodically on $\R$. For
$k\in\N$,  we set $\Phi_k(x):=\varphi_{t,q}(kx)$. Choosing $k\in\N$ large
enough, we set
\[
u(x):=u_t(x)+\int_0^x\d s\int_0^s\d\tilde s\, \Phi_k\left(\tilde s\right)\,.
\]
It is obvious that this function has all the desired properties (for
large enough $k$).
\end{proof}

\begin{lemma}
\label{lem:auxconst2}
  Let $\e>0$, $t\in[0,1]$ and let $\xi_1,\xi_2\in\Rnsym$ and $\alpha\in\R$, $\eta\in\R^n$ such
  that $\xi_1-\xi_2=\alpha \eta\otimes \eta$. Let $l:[0,1]^n\to\R$ be affine,
  and $u_t(x)=l(x)+\frac12
  x^T(t\xi_1+(1-t)\xi_2)x$. Then there exists
 a function $u:[0,1]^n\to \R$ and open sets $\Omega_1,\Omega_2\subset
 [0,1]^n$ such that
\[
\begin{split}
  \|u-u_t\|_{L^\infty}+\|u-u_t\|_{L^\infty}&<\e\\
  \nabla^2 u&= \xi_1\text{ on }\Omega_1\\
  \nabla^2 u&= \xi_2\text{ on }\Omega_2\\
\|\nabla^2 u\|_{L^\infty}&\leq C\\
  |\L^n(\Omega_1)-t|&<\e\\
  |\L^n(\Omega_2)-(1-t)|&<\e\,.
\end{split}
\]
\end{lemma}

\begin{proof}
  We may fill the cube $[0,1]^n$ by smaller cubes with one of the axes parallel
  to $\eta$, and set $u=u_t$ on the (small) remainder.  In this way, we reduce the problem
  to the case where $\eta=e_1$. Now let $\Omega_\e:=[\e,1-\e]^n$, and $\eta\in C_0^\infty((0,1)^n)$ such that
\[
\begin{split}
  \eta&=1\text{ on }\Omega_\e \\
\|\nabla\eta\|_{L^\infty}&\leq \frac{L}{\e}\\
  \|\nabla^2\eta\|_{L^\infty}&\leq \frac{L}{\e^2}\,,
\end{split}
\]
where $L$ is some numerical constant that does not depend on $\e$. For
$x\in[0,1]$, we write $\tilde u_t(s)=\frac12
  (t(\xi_1)_{11}+(1-t)(\xi_2)_{11})s^2$. Let $\delta>0$ to be chosen later. According to
Lemma \ref{lem:auxconst1}, we may choose 
$\tilde u:[0,1]\to\R$ and $I,J\subset[0,1]$ such that $\overline
I\cup\overline J=[0,1]$, $I\cap J=\emptyset$, $|I|=t$, $|J|=(1-t)$, and 
\[
\begin{split}
  \tilde u(0)=\tilde u_t(0),\quad \tilde u'(0)=\tilde u'_t(0),\\
  \tilde u(1)=\tilde u_t(1),\quad \tilde u'(1)=\tilde u'_t(1),\\
  \|\tilde u-\tilde u_t\|_{L^\infty}+\|\tilde u'-\tilde u_t'\|_{L^\infty}<\delta\\
  \tilde u''(x)= \xi_1 \quad\text{ for }x\in I\\
  \tilde u''(x)= \xi_2 \quad\text{ for }x\in J\,.
\end{split}
\]
We set $\psi(x_1,\dots,x_n)=\tilde u(x_1)$ and 
\[
u:=\eta (\psi+l)+(1-\eta) u_t\,.
\]
Choosing $\delta$ small enough (e.g. $\delta<\min(\e^3,\e^3/L)$), this choice of
$u$ satisfies all the requirements. We leave it to the reader to carry out the straightforward computations
that lead to 
this statement.
\end{proof}

\begin{lemma}
  \label{lem:RsymQ2}
Assume that $f:\Rnsym\to\R$ is bounded from above by a continuous function
$\tilde f\in C^0(\Rnsym)$. Then we have 
\[
Q_2 f\leq \Rs f\,.
\]
\end{lemma}
\begin{proof}
 Since $Q_2f$ is the largest 2-quasiconvex function that is less or equal to
 $f$, and $\Rs f$ is the largest symmetric rank one convex function that is less or equal to
 $f$, it suffices to show that if $g:\R^{n\times n}\to \R$ is 2-quasiconvex, and
 $g\leq f$,
 then it is symmetric-rank-one convex. So let us suppose $g:\R^{n\times n}\to\R$
 is 2-quasiconvex, and let $\xi_1,\xi_2\in\R^{n\times n}$ and $\alpha\in\R$,
 $\eta\in\R^n$ such that $\xi_1-\xi_2=\alpha\eta\otimes\eta$. We need to show
 \begin{equation}
   \label{eq:1}
   g(t\xi_1+(1-t)\xi_2)\leq t g(\xi_1)+(1-t) g(\xi_2) 
 \end{equation}
for all $t\in[0,1]$. Let $u_t(x):= \frac12 x^T(t\xi_1+(1-t)\xi_2)x$, and $\e>0$
to be chosen later. Let $u$
be the approximating function of Lemma \ref{lem:auxconst2}, with the sets
$\Omega_1,\Omega_2\subset[0,1]^n$ as in the statement of that lemma. Then we
have $u-u_t\in W^{2,\infty}_0([0,1]^n)$ and
\[
\nabla^2 u= t\xi_1+(1-t)\xi_2+\nabla^2(u-u_t)\,.
\]
Hence
\[
\begin{split}
  g(t\xi_1+(1-t)\xi_2)&=\int_{[0,1]^n}g(\nabla^2 u_t)\d x\\
&\leq \int_{[0,1]^n}g(\nabla^2 u)\d x\\
  &=\L^n(\Omega_1)g(\xi_1)+\L^n(\Omega_2)g(\xi_2)+
  \int_{[0,1]^n\setminus(\Omega_1\cup\Omega_2)} g(\nabla^2 u)\d x\\
&\leq \L^n(\Omega_1)g(\xi_1)+\L^n(\Omega_2)g(\xi_2)+
  \int_{[0,1]^n\setminus(\Omega_1\cup\Omega_2)} \tilde f(\nabla^2 u)\d x
\end{split}
\]
Choosing $\e$ small enough and using the properties of $u,\Omega_1,\Omega_2$
from the statement of Lemma \ref{lem:auxconst2}, we see that the right hand side is smaller than
$tg(\xi_1)+(1-t)g(\xi_2)+\delta$ for any given $\delta>0$; here we also used the
assumption that $\tilde f$ is continuous. This proves
\eqref{eq:1} and hence the lemma.
\end{proof}

\begin{definition} Let $f:\Rnsym\to\R$.
  We set $\Rs_0 f=f$ and
  \[
  \begin{split}
    \Rs_{k+1}f(\xi)&:=\inf\{ t
    \Rs_k(\xi_1)+(1- t )\Rs_k(\xi_2):\\
    &\qquad t \xi_1+(1- t )\xi_2=\xi,\xi_1-\xi_2=\alpha \eta\otimes \eta \text{
      for some }\alpha\in\R,\,\eta\in \R^n\}\,.
  \end{split}
  \]
\end{definition}
In complete analogy to Theorem 6.10, part 2 in \cite{MR2361288}, we  show

\begin{lemma}
\label{lem:Rsym}

Let $f:\Rnsym\to \R$, and let $g:\Rnsym\to \R$ be symmetric rank one
convex with $g\leq f$. Then we have
\[
\Rs f=\inf_{k\in\N}\Rs_k f\,.
\]
\end{lemma}

\begin{proof}
First we observe that for any $k\in\N$, we have
\[
g\leq \Rs_{k+1} f\leq \Rs_k f\,,
\]
and hence obtain that
\[
R'f:=\inf_{k\in \N}\Rs_k f=\lim_{k\to\infty}\Rs_k f
\]
is well defined. For any symmetric rank one convex function $g$ we have $R'g=g$,
and hence $R'(\Rs f)=\Rs f$. Furthermore, if $g\leq g'$, then $R'g\leq
R'g'$. Combining these observations with the fact $\Rs f\leq f$, we obtain
\[
\Rs f\leq R'f\leq f\,.
\] 
It remains to show that $R'f$ is symmetric rank one convex. Assume that
$\xi_1,\xi_2\in \Rnsym$, and $\alpha\in\R$, $\eta\in\R^{n}$ such that
$\xi_1-\xi_2=\alpha \eta\otimes\eta$. Let $\e>0$. By definition of $R'f$, there exist $i,j\in\N$
such that
\[
\Rs_if(\xi_1)\leq R'f(\xi_1)+\e,\quad \Rs_jf(\xi_2)\leq R'f(\xi_2)+\e\,.
\]
Without loss of generality, we may assume $i\leq j$, which yields
$\Rs_jf(\xi_1)\leq \Rs_if(\xi_1)$. Thus we obtain for every $t\in[0,1]$,
\[
\begin{split}
  R'f(t\xi_1+(1-t)\xi_2)&\leq \Rs_{j+1}(t\xi_1+(1-t)\xi_2)\\
  &\leq t\Rs_{j}(\xi_1)+(1-t)\Rs_j(\xi_2)\\
  &\leq tR'f(\xi_1)+(1-t)R'f(\xi_2)+\e\,.
\end{split}
\]
Since $\e>0$ was arbitrary, we obtain that $R'f$ is symmetric rank one convex,
which proves the lemma. 
\end{proof}

\begin{lemma}
\label{lem:RsGl}
  We have
\[
\Rs F_\lambda\leq \bar G_\lambda\,.
\]
\end{lemma}
\begin{proof}
From the definition of $\Rs F_\lambda$, we see that for $\xi\in \Rsym$, $R\in SO(2)$, we have 
\[
\Rs F_\lambda (R^T\xi  R)=\Rs F_\lambda (\xi)\,.
\]
Hence it suffices to consider $\xi$ of diagonal form,
\[
\xi=\left(\begin{array}{cc} x&0\\0&y\end{array}\right)\,.
\]
We may assume $|x|+|y|< \lambda$, since otherwise we know
$F_\lambda(\xi)=\bar G_\lambda(\xi)=\lambda+x^2+y^2$. Similarly, we may assume
$0<|x|+|y|$, since otherwise $F_\lambda(\xi)=\bar G_\lambda(\xi)=0$.
Let
$\alpha,\beta\in (0,1)$ to be chosen later, and set
\[
\xi_1=\left(\begin{array}{cc} 0&0\\0&0\end{array}\right)\,,\quad
\xi_2=\left(\begin{array}{cc} x/\alpha &0\\0&0\end{array}\right)\,,\quad
\xi_3=\left(\begin{array}{cc} x&0\\0&y/\beta\end{array}\right)\,.
\]
Note that $\beta \xi_3+(1-\beta)(\alpha \xi_2+(1-\alpha) \xi_1))=\xi$, and
$\xi_3-(\alpha\xi_2+(1-\alpha)\xi_1),\xi_2-\xi_1$ are both symmetric-rank-one.
By Lemma \ref{lem:Rsym}, we have
\[
\begin{split}
  \Rs F_\lambda(\xi)\leq &\beta  F_\lambda(\xi_3)+(1-\beta)\left(\alpha
    F_\lambda(\xi_2)+(1-\alpha) F_\lambda(\xi_1)\right)\\
  =&
  \beta\left(\lambda+x^2+\frac{y^2}{\beta^2}\right)+(1-\beta)\alpha\left(\lambda+\frac{x^2}{\alpha^2}\right)\,.
\end{split}
\]
Now we assume $|x|>0$. The right hand side in the last estimate is convex in $\alpha$; it attains its minimum at $\alpha=
\frac{|x|}{\sqrt{\lambda}}$. Hence, 
\[
\begin{split}
  \Rs F_\lambda(\xi)\leq &\beta
  \left(\lambda+x^2+\frac{y^2}{\beta^2}\right)+(1-\beta)2|x|\sqrt{\lambda}\\
  =& 2|x|\sqrt{\lambda}+\beta(\sqrt{\lambda}-|x|)^2+\frac{y^2}{\beta}
\end{split}
\]
Choosing $\beta=|y|/(\sqrt{\lambda}-|x|)$, we obtain
\[
\Rs  F_\lambda(\xi)\leq 2\sqrt{\lambda}(|x|+|y|-|xy|)=\bar G_\lambda(\xi)\,.
\]
It remains to prove the claim for the case $|x|=0$. Then we have 
\[
\begin{split}
  \Rs  F_\lambda(\xi)\leq &\beta  F_\lambda(\xi_3)+(1-\beta) F_\lambda(\xi_1)\\
  =&
  \beta\left(\lambda+x^2+\frac{y^2}{\beta^2}\right)\,.
\end{split}
\]
Again setting $\beta=|y|/(\sqrt{\lambda}-|x|)$, we obtain the same conclusion as before.
This proves the lemma.
\end{proof}

\begin{proof}[Proof of Theorem \ref{thm:appthm}]
By Theorem 6.28 in \cite{MR2361288}, we have $\bar G_1=Q_1F_1$. We have
$F_\lambda=\lambda F(\cdot/\sqrt{\lambda})$, and hence 
by the definition of
the quasiconvex envelope \eqref{eq:23} it is easily seen that
$Q_1F_\lambda= \lambda Q_1F_1(\cdot/\sqrt{\lambda})$. It is also easily verified that
$\bar G_\lambda=\lambda \bar G(\cdot/\sqrt{\lambda})$, and 
since
$Q_1F_\lambda\leq Q_2 F_\lambda$, we  obtain $\bar G_\lambda\leq
Q_2F_\lambda$. By Lemma \ref{lem:RsymQ2} and Lemma \ref{lem:RsGl}, we also have
the opposite inequality $\bar G_\lambda\geq Q_2F_\lambda$. This proves the theorem.
\end{proof}

\bibliographystyle{plain}
\bibliography{proposal}
\end{document}